\documentclass[a4paper,10pt]{article}
\usepackage{mathrsfs, amssymb}
\usepackage{amsthm}
\usepackage[intlimits]{empheq}
\usepackage{bbm} 
\usepackage{amsmath}
\usepackage{hyperref}
\usepackage{xcolor}
\usepackage[scientific-notation=true]{siunitx}
\usepackage{float}
\usepackage{dsfont}	
\usepackage[numbers]{natbib}
\usepackage[nottoc,numbib]{tocbibind}
\usepackage{placeins}
\usepackage{subcaption}
\usepackage{url}
\usepackage{listings}
\usepackage{authblk}
\usepackage[inner=3cm,outer=3cm,top=2cm,bottom=4cm]{geometry}
\usepackage{todonotes}
\newtheorem{theorem}{Theorem}[section]
\newtheorem{remark}{Remark}[section]
\newtheorem{example}[theorem]{Example}
\newtheorem{proposition}{Proposition}
\newtheorem{corollary}{Corollary}
\newtheorem{lemma}{Lemma}
\newtheorem{definition}{Definition}

\newcommand \dd[1] {\mathrm d{#1} }
\allowdisplaybreaks

\begin{document}

\title{On finding optimal collective variables for complex systems by minimizing the deviation between effective and full dynamics}
\date{}

\author[1]{Wei Zhang\,\thanks{Email: wei.zhang@fu-berlin.de}}
\author[1,2]{Christof Sch\"utte\,\thanks{Email: christof.schuette@fu-berlin.de}}

\affil[1]{Zuse Institute Berlin, Takustra{\ss}e 7, 14195 Berlin, Germany}

\affil[2]{Department of Mathematics and Computer Science, Freie Universit\"at Berlin, Arnimallee 6, 14195 Berlin, Germany}

\maketitle

\begin{abstract}
    This paper is concerned with collective variables, or reaction
    coordinates, that map a discrete-in-time Markov process $X_n$ in
    $\mathbb{R}^d$ to a (much) smaller dimension $k\ll d$. We define the
    effective dynamics under a given collective variable map $\xi$ as the best Markovian representation of $X_n$ under $\xi$. The novelty of the paper is that it gives strict criteria for selecting optimal collective variables via the properties of the effective dynamics. 
    In particular, we show that the transition density of the effective
    dynamics of the optimal collective variable solves a relative entropy minimization problem from certain family of densities to the transition density of $X_n$. 
    We also show that many transfer operator-based data-driven numerical
    approaches essentially learn quantities of the effective dynamics.
    Furthermore, we obtain various error estimates for the effective dynamics
    in approximating dominant timescales\,/\,eigenvalues and transition rates of the original process $X_n$ and how optimal collective variables minimize these errors. Our results contribute to the development of theoretical tools for the understanding of complex dynamical systems, e.g.\ molecular kinetics, on large timescales.  
    These results shed light on the relations among existing data-driven
    numerical approaches for identifying good collective variables, and they also motivate the development of new methods.
\end{abstract}


\section{Introduction}
\label{sec-intro}

 In many realistic applications, e.g.\ molecular dynamics (MD), materials
 science, and climate simulation, the systems of interest are often very high-dimensional
 and exhibit extraordinarily complex phenomena on vastly different time scales. 
For instance, it is common that molecular systems, e.g.\ proteins, undergo noise-induced transitions among several
long-lived (metastable) conformations, and these essential transitions occur on timescales that are much larger compared to the timescale of noisy fluctuations experienced by the systems. 
 Understanding the dynamics of such systems is  important for
 scientific discoveries, e.g.\ in drug design. However, the characteristics of these systems, e.g.\ their
 high dimensionality and the existence of multiple time/spatial scales, bring
 considerable challenges, as many traditional experimental/computational
 approaches become either inefficient or completely inapplicable. A great amount of research efforts have thus been devoted to developing theoretical tools as well as advanced numerical approaches for tackling these challenges. 

A commonly adopted strategy is to utilize the fact that the dynamics of
high-dimensional systems, e.g.\ molecular systems in MD and materials science,
can often be well characterized using only a few observables, often called
collective variables (CVs) or reaction coordinates, of the systems. In
particular, many enhanced sampling methods
(see~\cite{enhanced-sampling-for-md-review} for a review) utilize the
knowledge about the system's CVs to perform accelerated sampling and efficient
free energy calculations~\cite{lelievre2010free}, and various methods for
model reduction rely on this knowledge to construct simpler
(lower-dimensional) surrogate
models~\cite{perspective-noid-cg,effective_dynamics,non-markovian-modeling-pfolding-netz}.
While these methods have proven to be greatly helpful in understanding
dynamical behaviors of complex systems, the accuracy of estimated quantities
and surrogate models offered by these methods crucially depend on the choices
of CVs. Given a set of CVs, analyzing the accuracy of surrogate models in
approximating the true dynamics is an important research topic that has
attracted considerable attentions in the past years;
see~\cite{effective_dynamics,LEGOLL2017-pathwise,effective_dyn_2017,LegollLelievreSharma18,DLPSS18,LelievreZhang18}
for studies on effective dynamics constructed using conditional expectation.
These work extend the classical averaging technique for multiscale processes to
the setting where the slow variable is nonlinear
(see~\cite[Chapter~10]{pavliotis2008multiscale} for averaging techniques and
\cite{LelievreZhang18,LegollLelievreSharma18} for discussions on the connection
between averaging and effective dynamics by conditional expectation). 
Furthermore, efficient algorithms which allow for automatic identification of ``good" CVs are currently under rapid development; see for instance~\cite{zhang2023understanding,optimal-rc-bittracher,hao-rc-flow} for recent progresses in finding CVs using machine learning and deep learning techniques. 
 
As generating data for large and complex systems becomes much easier thanks to the significant advances in the developments of both computer hardware and efficient numerical algorithms in the past decades, many data-driven numerical methods have been proposed, which can be applied to understanding dynamical behaviors of complex systems by analyzing their trajectory data. A large class of these data-driven methods are based on the theory of transfer operators~\cite{transfer_operator,Schutte_Klus_Hartmann_2023} or Koopman operators~\cite{Koopmanism}, in which the dynamics of a underlying system is analyzed using the eigenvalues and the corresponding eigenfunctions of the transfer (or Koopman) operator associated to the system.
Notable examples of methods in this class include standard Markov state models (MSMs)~\cite{msm_generation,husic-pande-msm}, the variational approach to conformational dynamics~\cite{frank_feliks_mms2013,feliks_variational_jctc_2014}, time lagged independent component analysis (tICA)~\cite{tica}, variational approach for Markov processes (VAMP)~\cite{Wu2020}, extended dynamic mode decompositions~\cite{edmd,klus-koopman,EDMD-Klus}, and many others. Recent development in this direction includes kernel-tICA~\cite{kernel-tica}, as well as the deep learning frameworks VAMPnets~\cite{vampnet}, and state-free reversible VAMPnets (SRVs)~\cite{state-free-vampnets} for molecular kinetics.

 This paper is concerned with collective variables of discrete-in-time Markov
 processes $X_n$ in $\mathbb{R}^d$ in high dimension $d$, where $X_n$ may also
 be understood as resulting from a continuous-time Markov process by subsampling with constant lag-time.  Collective variables are described by CV maps $\xi:\mathbb{R}^d\rightarrow \mathbb{R}^k$, where $1 \le
k < d$, with an interest in $k\ll d$.  Following the idea
in~\cite{effective_dynamics}, we define the effective dynamics of $X_n$ under a given
CV map $\xi$ using conditional expectations. The main goal is to analyze how
good the effective dynamics is in approximating the original process $X_n$ on
large timescales, and also use this analysis to propose numerical approaches for
optimizing CV maps. For this purpose, we first lay out the theoretical basis
(in Section~\ref{sec-tran-operator}) by collecting results on how the largest (relaxation) timescales of $X_n$ are
connected to the dominant eigenvalues of the transfer operator, as well as how
transition rates between metastable sets can be computed using the committor
function via transition path theory
(TPT)~\cite{tpt_eric2006,towards_tpt2006,tpt2010}. These well-known results are reformulated in a variational setting that is instrumental for the subsequent analysis. 

\paragraph{Main results} The novelty of the paper is that it gives strict criteria for selecting \emph{optimal} collective variables via the properties of the effective dynamics. Concretely, we give three optimality statements: First, we show that the effective dynamics introduced in this paper (Definition~\ref{def-effdyn}) is the optimal surrogate model in the sense that the Kullback-Leibler (KL) ``distance" to the transition density of $X_n$ is minimized (Proposition~\ref{prop-eff-kl}). This optimality allows us to define the optimal CV map as the solution to a relative entropy minimization problem (Theorem~\ref{thm:eff-kl}). Next, we obtain quantitative error estimates for the effective dynamics in approximating the eigenvalues and transition rates of the original process. Based on these
results, optimal CVs for minimizing these approximation errors are characterized, see Theorems~\ref{thm:opt_timescales} and \ref{thm:rates_CV}. 

\paragraph{Algorithmic implications} We show that our theoretical results also impact transfer operator-based data-driven numerical approaches.
For instance, in order to reduce computational complexity or pertain certain properties, e.g.\ rotation and translation invariance in molecular dynamics,
many data-driven numerical algorithms such as the deep learning approach VAMPnets for learning eigenfunctions~\cite{vampnet,state-free-vampnets}, often employ features, e.g., internal variables, of the underlying systems (instead of working directly with coordinate data) for function representation.
We show that in the case where such system's features are employed, the numerical methods implicitly learn the quantities (e.g.\ timescales, transition rates) of the effective dynamics, see Remark~\ref{vamp-r-eff} for formal discussions.
Moreover, the result on the relative entropy minimization property of the effective dynamics of the optimal CV uncovers a close connection to the transition manifold framework~\cite{Bittracher2018}, especially to its variational formulation \cite{optimal-rc-bittracher}, as well as to the recent deep learning approach for identifying CVs using normalizing flows~\cite{hao-rc-flow} (see
Theorem~\ref{thm:eff-kl} and Remark~\ref{rmk-prop-eff-kl}). 

The remainder of this article is organized as follows. In
Section~\ref{sec-tran-operator}, we study the transfer operator approach for analyzing metastable Markov processes. We present variational characterizations of the system's main timescales via spectral analysis of the associated transfer operator, and of transition rates using TPT. In
Section~\ref{sec-effdyn}, we define the effective dynamics of discrete-in-time Markov processes for a given CV map and study its properties, in particular the relative entropy minimization property of the optimal CV. 
In Section~\ref{sec-error-estimates}, we analyze approximation errors of effective dynamics in terms of estimating eigenvalues and transition rates of the original process, and discuss algorithmic identification of optimal CVs.
In Appendix~\ref{app-sec-langevin}, we provide details for the concrete case where the transfer
operator models the evolution of position variables of Langevin dynamics. In
Appendix~\ref{app-sec-proofs}, we present the proofs of some results in Section~\ref{sec-tran-operator}.

\section{Transfer operator approach}
\label{sec-tran-operator}

In this section, we present the transfer operator approach for Markov processes. 
We introduce necessary quantities in Section~\ref{subsec-transfer-definition},
then we study the spectrum of the transfer operator in Section~\ref{subsec-spectrum}, and finally we present an extension of TPT in Section~\ref{subsec-tpt}.  This section provides theoretical basis for the study of effective dynamics in Sections~\ref{sec-effdyn}--\ref{sec-error-estimates}.

\subsection{Definitions and basic properties}
\label{subsec-transfer-definition}

Consider a stationary, discrete-in-time Markov process $X_n$ in $\mathbb{R}^d$
with a positive transition density function $p(x,y)$ for $x,y \in \mathbb{R}^d$, i.e.\
$\mathbb{P}(X_{n+1}\in D| X_n=x) = \int_{D} p(x,y)\,\dd{y}$ for any measurable subset $D
\subseteq \mathbb{R}^d$ and any integer $n\ge 0$.  We assume that $X_n$ is
ergodic with respect to a unique invariant distribution $\mu(\dd{x}) =
\pi(x)\dd{x}$, where the probability density $\pi$ is positive
(see~\cite[Chapter 1]{Krengel+1985} and \cite[Section 2]{Eberle-note} for ergodicity of Markov processes). 
The invariance of~$\mu$ means
\begin{equation}
  \left(\int_{\mathbb{R}^d} p(x,y) \mu(\dd{x})\right) \dd{y}= \mu(\dd{y}), ~\mbox{or equivalently,}~ \int_{\mathbb{R}^d} p(x,y) \mu(\dd{x}) = \pi(y), \quad  y \in \mathbb{R}^d\,.
  \label{eqn-p-pi}
\end{equation}

A key object in analyzing the dynamical properties of the process $X_n$ is the transfer operator $\mathcal{T}$, which is defined by its action on functions $f: \mathbb{R}^d\rightarrow \mathbb{R}$, i.e.
\begin{equation}
  (\mathcal{T}f)(x) = \mathbb{E}(f(X_1)|X_0=x) = \int_{\mathbb{R}^d} p(x,y) f(y)\,\dd{y} \,, \quad x \in \mathbb{R}^d\,,
  \label{trans-operator}
\end{equation}
where $\mathbb{E}(\cdot|X_0=x)$ denotes the expectation conditioned on the
initial state $X_0=x$ (see Remark~\ref{rmk-operators} below for discussions on different operators studied in the literature as well as their connections).
It is natural to study $\mathcal{T}$ in the Hilbert space
\begin{equation*}
\mathcal{H}=L^2_\mu(\mathbb{R}^d)=\left\{f\,\middle|\, f:\mathbb{R}^d\rightarrow
\mathbb{R},\, \int_{\mathbb{R}^d} f^2(x) \mu(\dd{x}) < +\infty\right\} 
\end{equation*}
  endowed
with the weighted inner product (and the norm denoted by $\|\cdot\|_\mu$)
\begin{equation}
  \langle f, h\rangle_\mu= \int_{\mathbb{R}^d} f(x)h(x) \mu(\dd{x}) \,, \quad \forall f,h \in \mathcal{H}\,.
  \label{weighted-inner-product}
\end{equation}

With the above preparation we study the reversibility of $X_n$.
Denote by $\mathcal{T}^*$ the adjoint of $\mathcal{T}$ in $\mathcal{H}$, such
that $\langle \mathcal{T}f, h\rangle_\mu= \langle f, \mathcal{T}^*h\rangle_\mu$, for all $f,h \in \mathcal{H}$. It is straightforward to verify that
\begin{equation}
  (\mathcal{T}^*f)(x) = \int_{\mathbb{R}^d} p^{*}(x,y) f(y)\, \dd{y} \,, \quad x \in \mathbb{R}^d 
  \label{tran-adjoint}
\end{equation}
for test functions $f: \mathbb{R}^d\rightarrow \mathbb{R}$, where $p^*$ is defined by 
\begin{equation}
  p^*(x,y) = \frac{p(y,x)\pi(y)}{\pi(x)}\,, \quad x,y \in \mathbb{R}^d\,.
  \label{p-adjoint}
\end{equation}
Moreover, from \eqref{eqn-p-pi} and \eqref{p-adjoint} we can easily verify that
$\int_{\mathbb{R}^d} p^*(x,y)\,\dd{y} = 1$ for all $x\in \mathbb{R}^d$ and
$\int_{\mathbb{R}^d} p^*(x,y)\,\mu(\dd{x})=\pi(y)$ for all $y\in
\mathbb{R}^d$. Therefore, $p^*$ is a transition density function and it defines another Markov process (i.e.\ the adjoint of $X_n$) which has the same invariant distribution $\mu$.

Note that $\mathcal{T}$ can be decomposed into the reversible and non-reversible parts, i.e.
\begin{equation}
  \mathcal{T} = \mathcal{T}^{rev} + \mathcal{T}^{non}
  \label{tran-decomp-1}
\end{equation}
where 
\begin{equation}
     \mathcal{T}^{rev} = \frac{1}{2} (\mathcal{T} + \mathcal{T}^*), \quad \mathcal{T}^{non} = \frac{1}{2} (\mathcal{T} - \mathcal{T}^*),
  \label{tran-decomp-2}
\end{equation}
or, more explicitly,
\begin{equation}
  \begin{aligned}
  (\mathcal{T}^{rev}f)(x) =& \int_{\mathbb{R}^d}
    \frac{1}{2}\big(p(x,y)+p^*(x,y)\big) f(y)\, \dd{y}\,, \\
    (\mathcal{T}^{non}f)(x) =& \int_{\mathbb{R}^d} \frac{1}{2}\big(p(x,y)-p^*(x,y)\big) f(y)\,\dd{y}
  \end{aligned}
  \label{tran-decomp-3}
\end{equation}
for test functions $f: \mathbb{R}^d\rightarrow \mathbb{R}$. From
\eqref{tran-adjoint}, \eqref{p-adjoint}, and the simple decomposition in
\eqref{tran-decomp-1}--\eqref{tran-decomp-3}, we obtain the following equivalent conditions.
\begin{align}
  \begin{split}
    & ~ p^*\equiv p \hspace{2cm} \mbox{(detailed balance)} \\
    \Longleftrightarrow & ~ \mathcal{T}=\mathcal{T}^*= \mathcal{T}^{rev} \hspace{0.6cm} \mbox{(self-adjointness)}\\
    \Longleftrightarrow & ~ \mathcal{T}^{non}=0\,.
  \end{split}
  \label{reversibility}
\end{align}
We will refer to \eqref{reversibility} as reversibility conditions and say that $X_n$ is reversible if any of the equivalent conditions in \eqref{reversibility} is satisfied. 

Let us define the Dirichlet form~\cite[Chapter~1]{FukushimaOshimaTakeda2010} 
\begin{equation}
  \mathcal{E}(f,h) := \frac{1}{2} \int_{\mathbb{R}^d}\left(\int_{\mathbb{R}^d}
   (f(y) - f(x)) (h(y)- h(x)) p(x,y)\dd{y}\,\right) \mu(\dd{x})\,,
  \label{dirichlet-form}
\end{equation}
for two test functions $f,h: \mathbb{R}^d\rightarrow \mathbb{R}$, and the Dirichlet energy 
\begin{equation}
\mathcal{E}(f) := \mathcal{E}(f,f) = \frac{1}{2} \int_{\mathbb{R}^d}\left(\int_{\mathbb{R}^d}
   (f(y) - f(x))^2 p(x,y)\dd{y}\,\right) \mu(\dd{x})\,.
  \label{energy}
\end{equation}
The following lemma summarizes basic properties of these quantities. 
\begin{lemma}
  For any $f,h \in \mathcal{H}$, we have 
\begin{equation}
  \mathcal{E}(f,h) = \langle (\mathcal{I} - \mathcal{T}^{rev})f, h\rangle_{\mu}\,,
  \label{e-f-g-adjoint}
\end{equation}
  where $\mathcal{I}$ and $\mathcal{T}^{rev}$ denote the identity operator and the reversible part of $\mathcal{T}$ in \eqref{tran-decomp-2}, respectively. Let $X_0,
  X_1,X_2, \dots$ be an infinitely long trajectory of the process $X_n$. Then, we have almost surely
\begin{equation}
  \mathcal{E}(f) = \lim_{N\rightarrow +\infty} \frac{1}{2N} \sum_{n=0}^{N-1} |f(X_{n+1}) - f(X_n)|^2\,. 
  \label{msv-is-energy}
\end{equation}
  \label{lemma-energy-and-msv}
\end{lemma}
\begin{proof}
  Concerning \eqref{e-f-g-adjoint}, from \eqref{dirichlet-form} we can derive 
  \begin{align*}
    \mathcal{E}(f,h) =& \frac{1}{2} \int_{\mathbb{R}^d}\left(\int_{\mathbb{R}^d}
     \Big(f(y)h(y) + f(x)h(x) - f(y)h(x) - f(x)h(y)\Big) \,p(x,y)\dd{y}\right) \mu(\dd{x}) \\
    =& \int_{\mathbb{R}^d} f(x)h(x) \,\mu(\dd{x}) - \frac{1}{2}
    \int_{\mathbb{R}^d}\left(\int_{\mathbb{R}^d}  \Big(f(y)h(x) + f(x)h(y)\Big) p(x,y)\dd{y}\right) \mu(\dd{x}) \\
    =& \int_{\mathbb{R}^d} f(x)h(x) \,\mu(\dd{x}) - \frac{1}{2} \Big(\langle \mathcal{T}f,h\rangle_\mu +\langle f, \mathcal{T}h\rangle_\mu\Big) \\
    =& \langle (\mathcal{I} - \mathcal{T}^{rev})f,h\rangle_\mu 
  \end{align*}
  where we used \eqref{eqn-p-pi},\eqref{trans-operator} and \eqref{tran-decomp-2} to obtain the second, the third and the last equality, respectively.
Since we assume that the process is ergodic, the identity \eqref{msv-is-energy}
  follows directly from \eqref{energy} and Birkhoff's ergodic theorem~(see~\cite[Theorem~4.4 in Chapter 1]{Krengel+1985}). 
\end{proof}

In the following remark we discuss the relations among different operators that were studied in the literature.
\begin{remark}[Alternative operators and terminologies]
   Different operators have been used under different names in the literature for the study of stochastic dynamical systems.  The transfer
   operator $\mathcal{T}$ in \eqref{trans-operator} is also called backward transfer operator \cite{transfer_operator} or Koopman
   operator~\cite{Wu2020,Schutte_Klus_Hartmann_2023,KLUS2020132416}, whereas
   its adjoint $\mathcal{T}^*$ in \eqref{tran-adjoint} is called forward
   transfer operator in \cite{transfer_operator}. In the reversible case,
   we have $\mathcal{T}=\mathcal{T}^*$ and hence all these operators coincide with
   each other. In this work, we focus on the operator $\mathcal{T}$ and stick to the name transfer operator. The analysis can be straightforwardly applied to Koopman operator and
   forward transfer operator thanks to the relations among these operators.
   \label{rmk-operators}
\end{remark}

We conclude this section with a discussion on the concrete examples that are most interesting to this work.

\begin{example}
  Transfer operator approach is often used as a theoretical tool in data-driven numerical methods in order to analyze time-series
  $X_0=x(0), X_1=x(\tau), X_2=x(2\tau),\dots, X_{N-1}=x((N-1)\tau)$ of a system that is governed by a underlying SDE
  \begin{equation}
    \dd{x}(t) = b(x(t))\,\dd{t} + \sigma(x(t))\,\dd{\eta}(t)\,,\quad t\ge 0\,,
    \label{sde}
  \end{equation}
  with some drift and diffusion coefficients $b:\mathbb{R}^d\rightarrow
  \mathbb{R}^d$ and $\sigma: \mathbb{R}^d\rightarrow \mathbb{R}^{d\times d}$,
  respectively, where $\tau>0$ is called the lag-time~\cite{msm_generation},
  $x(t)\in \mathbb{R}^d$ denotes the system's state at time~$t$, and $\eta(t)$ is a $d$-dimensional Brownian motion.
  In this case, the transition density $p(x,y)$ of $X_n$ is related to the
  transition density $p^{\mathrm{SDE}}(t,y\,|x)$ of $x(t)$ in \eqref{sde}
  at time $t=\tau$, starting from $x\in \mathbb{R}^d$ at time $t=0$. The latter
  solves the Fokker-Planck equation~\cite{stoch_process_pavliotis,risken1989fpe}
    \begin{align}
      \begin{split}
	&\frac{\partial p^{\mathrm{SDE}}}{\partial t} = \mathcal{L}^\top p^{\mathrm{SDE}}, \quad t > 0\,,\\
	& p^{\mathrm{SDE}}(0,y\,|x) = \delta(x-y),\quad t = 0
      \end{split}
      \label{kolmogrov}
    \end{align}
    for $t\in (0,+\infty)$ and $y\in \mathbb{R}^d$, where the operator $\mathcal{L}^\top$, given by $\mathcal{L}^\top f=-\mathrm{div}(bf)
    + \frac{1}{2} \sum_{i,j=1}^d\frac{\partial^2 \big((\sigma\sigma^\top)_{ij}f\big)}{\partial x_ix_j}$ for a test function $f$,
    is the adjoint (with respect to the Lebesgue measure) of the generator of \eqref{sde}, and $\delta(\cdot)$ denotes the Dirac delta function.
  In many cases, the solution to \eqref{kolmogrov} is smooth for $t>0$ (see~\cite[Chapter 4]{stoch_process_pavliotis} and \cite{bogachev2022fokker} for a comprehensive study).
  In the following, we consider the transition density $p(x,y)$ in three concrete examples.
  \begin{enumerate}
    \item Brownian dynamics. Assume that SDE \eqref{sde} has the form 
  \begin{equation}
    \dd{x}(t) = -\nabla V(x(t))\,\dd{t} + \sqrt{2\beta^{-1}}\dd{\eta}(t)\,, \quad t\ge 0
    \label{sde-overdamped}
  \end{equation}
      where $V: \mathbb{R}^d\rightarrow \mathbb{R}$ is a smooth potential function and $\beta>0$ is a constant related to the system's temperature.
      The transition density of $X_n$ is then $p(x,y)=p^{\mathrm{BD}}(\tau,y\,|x)$,
      where $p^{\mathrm{BD}}(t,y\,|x)$ is the solution to \eqref{kolmogrov}
      with the generator $\mathcal{L}f =-\nabla V\cdot \nabla f +
      \frac{1}{\beta} \Delta f$.  Under certain conditions on $V$, $X_n$ is
      ergodic with the unique invariant distribution $\mu$ whose density is
      given by $\pi(x) = \frac{1}{Z} \mathrm{e}^{-\beta V(x)}$, where $Z=\int_{\mathbb{R}^d} \mathrm{e}^{-\beta V(x)}\,\dd{x}$ is a normalizing constant.
      In particular, $X_n$ is reversible and the corresponding transfer
      operator defined in \eqref{trans-operator} is given by
      $\mathcal{T}f=\mathrm{e}^{\tau \mathcal{L}}f$ (the solution to the
      backward Kolmogorov equation associated to $\mathcal{L}$~\cite[Chapter 8]{oksendalSDE}).
    \item Langevin dynamics. Assume that the evolution of $x(s)$ is governed by the (underdamped) Langevin dynamics~\cite[Section 2.2.3]{lelievre2010free}
  \begin{align}
    \begin{split}
      \dd{x}(t) =& v(t)\dd{t}\\
      \dd{v}(t) =& -\nabla V(x(t))\dd{t} - \gamma\,v(t)\dd{t} + \sqrt{2\gamma\beta^{-1}}\dd{\eta}(t)
    \end{split}
    \label{sde-langevin}
  \end{align}
      where $v(t)\in \mathbb{R}^d$ denotes the system's velocity and $\gamma>0$ is
      a friction constant. In this case, if we make the assumption that the (position) data $X_0, X_1, X_2,\dots$ comes from a Markov process and estimate its transition density using a density estimation method, then the ideal transition density (obtained as data size goes to infinite) is given by 
      \begin{equation}
      p(x,y) =
	(\frac{\beta}{2\pi})^{\frac{d}{2}}\int_{\mathbb{R}^d}\int_{\mathbb{R}^d}
	p^{\mathrm{Lan}}(\tau,y,v'|x,v)\, \mathrm{e}^{-\frac{\beta
	|v|^2}{2}}\,\dd{v}\,\dd{v'}\,,
	\label{p-langevin}
      \end{equation}
	where $p^{\mathrm{Lan}}(\tau,y,v'|x,v)$ denotes the transition density of \eqref{sde-langevin}.
      We refer to Appendix~\ref{app-sec-langevin} for detailed derivations.
      In particular, we show that the Markov process described by $p(x,y)$ in \eqref{p-langevin} is reversible, although the original Langevin dynamics is non-reversible.
\item Numerical scheme. Assume that the time-series data comes from sampling
  the SDE \eqref{sde-overdamped} using a discretization scheme 
  \begin{equation}
    X_{n+1} = X_n -\nabla V(X_n)\,\Delta t + \sqrt{2\beta^{-1}\Delta t}\,W_n, \quad n=0, 1, \dots, 
    \label{euler-maruyama}
  \end{equation}
      where $(W_n)_{n\ge 0}$ are independent standard Gaussian variables in $\mathbb{R}^d$ and $\Delta t$ is the step-size (here, for simplicity we assume Euler-Maruyama scheme~\cite[Chapter 9]{kloeden2011numerical} and $\tau=\Delta t$). In this case, the transition density of
      $X_n$ is given by $p(x,y) = \big(\frac{\beta}{4\pi\Delta
      t}\big)^{\frac{d}{2}}\exp(-\frac{\beta |y-x+\nabla V(x)\Delta t|^2}{4\Delta t})$, for $x,y\in \mathbb{R}^d$.
      The process $X_n$ is in general not reversible, but in practice this
      irreversibility due to numerical discretization is often ignored whenever $\Delta t$ is small.
      Alternatively, reversibility can be retained by adding a Matropolis-Hasting acceptance/rejection step~\cite[Chapter 2]{lelievre2010free}.
  \end{enumerate}
  \label{example-concrete}
\end{example}

\subsection{Spectrum and timescales} 
\label{subsec-spectrum}
In this section, we study the spectrum of $\mathcal{T}$ as well as its implications on the speed of convergence of the dynamics $X_n$ to equilibrium.

We assume that $X_n$ is reversible. By \eqref{reversibility}, $\mathcal{T}$ is
self-adjoint with respect to the inner product \eqref{weighted-inner-product}, and therefore its spectrum $\sigma(\mathcal{T})$ is real. 
To simplify the discussion, we further confine ourself to the case where
each nonzero element of $\sigma(\mathcal{T})$ is an eigenvalue $\lambda\in \mathbb{R}$, for which the (integral) equation 
\begin{equation}
  \mathcal{T}\varphi = \lambda \varphi\,,
  \label{eigen-problem}
\end{equation}
has a nonzero solution (eigenfunction). In the following proposition, we make our assumptions precise and give a characterization of the spectrum $\sigma(\mathcal{T})$. 

\begin{proposition}
  Assume that the transition density $p(x,y)$ is positive for all $x, y\in \mathbb{R}^d$, the operator $\mathcal{T}$ is self-adjoint with respect to
  \eqref{weighted-inner-product}, and the invariant density $\pi$ is positive. Also assume that 
  \begin{equation}
    \int_{\mathbb{R}^d} \int_{\mathbb{R}^d} p(x,y) p(y,x)\, \dd{x}\,\dd{y} < +\infty\,.
    \label{condition-hilbert-schmidt}
  \end{equation}
  Then, the following claims hold.
  \begin{enumerate}
    \item
      $\mathcal{T}$ is a compact operator on $\mathcal{H}$. Any nonzero value in $\sigma(\mathcal{T})$ is an eigenvalue of \eqref{eigen-problem}.
    \item
  Any eigenvalue $\lambda\in \mathbb{R}$ of \eqref{eigen-problem} satisfies that $-1 < \lambda \le 1$. Moreover, the only eigenfunction corresponding to $\lambda=1$ is the constant function.
  \end{enumerate}
  \label{prop-spectrum}
\end{proposition}

The proof of Proposition~\ref{prop-spectrum} is presented in Appendix~\ref{app-sec-proofs}.
The next proposition gives conditions which guarantee that all eigenvalues are nonnegative.

\begin{proposition}
  Under the assumptions in Proposition~\ref{prop-spectrum}, the following conditions are equivalent.
	  \begin{enumerate}
	    \item[(a)]
	     All eigenvalues of $\mathcal{T}$ are nonnegative.
	    \item[(b)]
	     The Markov chain $X_n$~satisfies $\mathbb{E}_{X_0\sim \mu}(f(X_0)f(X_1)) \ge 0$, for any $f \in \mathcal{H}$. 
	    \item[(c)]
	    There is a unique self-adjoint linear operator $\mathcal{A}$ on $\mathcal{H}$ such that $\mathcal{A} \ge 0$ and $\mathcal{T}=\mathcal{A}^2$.
      \end{enumerate}
      \label{prop-nonnegative-spectrum}
\end{proposition}
\begin{proof}
  Note that by the definition of $\mathcal{T}$ in \eqref{trans-operator} we have 
  \begin{align}
    \begin{split}
    \langle \mathcal{T}f, f\rangle_\mu =& \int_{\mathbb{R}^d}
      \Big(\mathbb{E}(f(X_1)|X_0=x)\Big) f(x) \mu(\dd{x})  = \mathbb{E}_{X_0\sim \mu}(f(X_0)f(X_1))\,, ~ f \in \mathcal{H}\,,
    \end{split}
    \label{proof-nonnegative-1}
  \end{align}
  where the second equality follows from the law of total expectation.
  Therefore, the equivalence between (a) and (b) follows from the variational
  principle for the smallest eigenvalue of $\mathcal{T}$ (see~\cite[Theorem~2.20]{teschl2009mathematical}).
  To prove the equivalence between (b) and (c), we notice that \eqref{proof-nonnegative-1} implies that (b) is equivalent to $\mathcal{T} \ge 0$. 
  Also, it is easy to verify that $\mathcal{T}$ is bounded. Therefore, applying the square root lemma~\cite[Theorem VI.9]{reed1981functional} we see that (b) implies (c).
  Conversely, assuming (c), we have $\langle \mathcal{T}f,f\rangle_\mu=
  \langle \mathcal{A}^2f, f\rangle_\mu = \langle \mathcal{A}f, \mathcal{A} f\rangle_\mu \ge 0$ and therefore (b) holds. 
  This shows that (b) and (c) are equivalent.
\end{proof}

In the following remark we discuss the conditions in Proposition~\ref{prop-nonnegative-spectrum}.
\begin{remark}[Positivity of eigenvalues]
  Roughly speaking, the condition (b) of
  Proposition~\ref{prop-nonnegative-spectrum} describes that ``on average, $f(X_n)$ does not change sign in two consecutive steps''.

  A situation under which the condition (c) is met is when the chain $X_n$ can
  be embedded into another reversible Markov process $Y_m$, such that $X_n=Y_{2n}$ for all $n\ge 0$.
  In particular, suppose that $X_n$ corresponds to the state of a reversible diffusion process governed by a SDE with a generator $\mathcal{L}$ 
  and let $\tau>0$ be the lag-time (see the first case in Example~\ref{example-concrete}). Then, $\mathcal{T}=\mathrm{e}^{\tau \mathcal{L}}= \mathcal{A}^2$ with $\mathcal{A}=\mathrm{e}^{\frac{\tau}{2} \mathcal{L}}$. In fact, in this case all eigenvalues of $\mathcal{T}$ are positive.

  In general, it is possible that a transfer operator $\mathcal{T}$ has negative eigenvalues $\lambda\in (-1, 0)$ under the assumptions in Proposition~\ref{prop-spectrum}.
  Such negative eigenvalues (and the corresponding eigenfunctions) may even be important in studying the large-time dynamics of $X_n$, if they are close to $-1$.
  Similar to Theorem~\ref{thm-variational-form-transfer} below, variational characterization for those negative eigenvalues (which are close to $-1$)
  can be obtained and applied to developing numerical methods.  However, we will not
  consider negative eigenvalues, since we are mainly interested in the case where the Markov chains come from reversible SDEs and the eigenvalues in this case are all positive according to the discussion above.
  \label{rmk-Positivness-eigenvalues}
\end{remark}

Let us assume that the assumptions in Proposition~\ref{prop-spectrum}, as well as one of the conditions in Proposition~\ref{prop-nonnegative-spectrum}, are satisfied.
In this case, the eigenvalues $(\lambda_i)_{i\ge 0}$ to the problem \eqref{eigen-problem} are nonnegative, and they can be ordered in a way such that  
\begin{equation}
  1=\lambda_0 > \lambda_1 \ge \lambda_2 \ge \dots\, \ge 0\,.
  \label{eigen-seq}
\end{equation}
Moreover, by Hilbert-Schmidt theorem~\cite[Theorem VI.16]{reed1981functional},
we can choose the corresponding eigenfunctions $(\varphi_i)_{i\ge 0}$, with
$\varphi_0\equiv 1$, such that they form an orthonormal basis of $\mathcal{H}$.
For a function $f\in \mathcal{H}$, from \eqref{trans-operator} we obtain  
\begin{equation}
  \begin{aligned}
    \mathbb{E}(f(X_n)|X_0=x) =& \mathbb{E}\big(\mathbb{E}(f(X_n)|X_{n-1})|X_0=x\big)
    = \mathbb{E}((\mathcal{T}f)(X_{n-1})|X_0=x) \\
    =& \dots = (\mathcal{T}^nf)(x) = \mathbb{E}_\mu(f) +
  \sum_{i=1}^{+\infty} \lambda_i^n \langle f, \varphi_i\rangle_\mu \varphi_i(x)\,,\quad n \ge 0\,,
  \end{aligned}
  \label{eqn-evolution-f-by-eigenfunc}
\end{equation}
where $\mathbb{E}_{\mu}(f)$ denotes the mean value of $f$ under $\mu$. Therefore, as the Markov chain evolves, the expectation
$\mathbb{E}(f(X_n)|X_0=x)$ converges to $\mathbb{E}_\mu(f)$ (i.e.\ the value at equilibrium), and the speed of convergence is determined
by the leading eigenvalues of~$\mathcal{T}$ that are closest to $1$. Estimating these leading eigenvalues and computing the corresponding eigenfunctions help
to understand the long-time dynamical behaviors (e.g.\ identifying large timescales and metastable states) of the process $X_n$.
In this regard, let us present the following variational characterization of the large eigenvalues and the corresponding eigenfunctions of $\mathcal{T}$ (see \cite[Theorem~2.1]{zhang2023understanding}). 

  \begin{theorem}
    Assume that the spectrum of $\mathcal{T}$ consists of discrete eigenvalues
    $(\lambda_i)_{i\ge 0}$ in \eqref{eigen-seq} with the corresponding (normalized) eigenfunctions $(\varphi_i)_{i\ge 0}$.
Given an integer $m \ge 1$ and positive constants $\omega_1 \ge \dots \ge \omega_m >0$, we have 
  \begin{align}
    \sum_{i=1}^m \omega_i(1-\lambda_i)=\min_{f_1,\dots, f_m\in\mathcal{H}} \sum_{i=1}^m \omega_i \mathcal{E}(f_i)\,,
    \label{variational-transfer-all-first-m}
  \end{align}
    where $\mathcal{E}$ is the Dirichlet energy in \eqref{energy}, and the
    minimization is over all $(f_i)_{1\le i \le m}$ in $\mathcal{H}$ under the constraints 
  \begin{equation} 
    \begin{aligned}
      \mathbb{E}_\mu (f_i) = 0, \quad \langle f_i,f_j\rangle_\mu = \delta_{ij}\,,  \quad \forall\, i,j\, \in \{1,\dots,m\}\,.
    \end{aligned}
    \label{f-constraints}
  \end{equation}
    Moreover, the minimum of
    \eqref{variational-transfer-all-first-m}--\eqref{f-constraints} is achieved when $f_i=\varphi_i$ for $1 \le i \le m$.
  \label{thm-variational-form-transfer}
\end{theorem}
We mention that similar variational characterizations for eigenvalues of generator have been obtained in~\cite[Theorem~3]{sz-entropy-2017} and~\cite[Theorem~1]{eigenpde-by-nn-ZHANG}. 
The proof of Theorem~\ref{thm-variational-form-transfer}, which we present in
Appendix~\ref{app-sec-proofs}, is adapted from the proof of \cite[Theorem~1, in Appendix A]{eigenpde-by-nn-ZHANG}.

In the following remark we discuss the connection between Theorem~\ref{thm-variational-form-transfer} and the variational formulations in VAMP~\cite{Wu2020}. 
Despite the similarities, we emphasize that, by employing the Dirichlet energy $\mathcal{E}$, we are able to draw a connection between eigenfunctions and slow CVs using Theorem~\ref{thm-variational-form-transfer} (see
Theorem~\ref{thm-optimal-cv-map} and the discussions after
Remark~\ref{rmk-vamp}), and obtain a unified framework that characterizes both
the spectrum and transition rates (see Proposition~\ref{prop-committor} in Section~\ref{subsec-tpt}) of the process. 

\begin{remark}[Connections to VAMP]
Theorem~\ref{thm-variational-form-transfer} is similar to the variational formulation in VAMP in the reversible case (see~\cite[Appendix D]{Wu2020}
  and~\cite[Section~II.A]{var-koopman-model}). In fact, using the identity
  \eqref{e-f-g-adjoint} in Lemma~\ref{lemma-energy-and-msv} and the normalization constraints in \eqref{f-constraints}, we can write \eqref{variational-transfer-all-first-m} as 
  \begin{align}
    \sum_{i=1}^m \omega_i\lambda_i=\max_{f_1,\dots, f_m\in\mathcal{H}} \sum_{i=1}^m \omega_i \langle f_i, \mathcal{T}f_i\rangle_\mu\,,
    \label{vamp-1}
  \end{align}
  which recovers the VAMP-1 score defined in~\cite[Theorem~2]{Wu2020} when $\omega_i=1$ for
  $1 \le i \le m$ (the eigenvalue $\lambda_0=1$ and its eigenfunction are
  excluded by imposing the first constraint in \eqref{f-constraints}). 
  \label{rmk-vamp}
\end{remark}

Based on Theorem~\ref{thm-variational-form-transfer}, we discuss the optimality of eigenfunctions as system's CV maps.
We also refer to Remark~\ref{rmk-opt-xi-timescales} for discussions on a numerical approach (based on Theorem~\ref{thm-variational-form-transfer} and the effective dynamics in the next section) for learning CV maps that can parametrize eigenfunctions.

To explain the idea, let us first assume $k=1$ and consider scalar CV maps.
Our observation is that, for a CV map $\xi:\mathbb{R}^d\rightarrow \mathbb{R}$
to be indicative of essential transitions of $X_n$ on large timescales, it is
necessary that $\xi$ is insensitive to the noisy fluctuations on the smallest timescales. In
other words, $\xi$ should be a slow variable such that the observations $\xi(X_0), \xi(X_1),\dots$ vary slowly (in most of the time unless essential transitions occur) as the dynamics $X_n$ evolves.
This reasoning suggests that a natural way to find a good CV map $\xi$ is by minimizing the average of squared variations
\begin{equation}
  \lim_{N\rightarrow +\infty} \frac{1}{2N} \sum_{n=0}^{N-1} |f(X_{n+1}) - f(X_n)|^2
  \label{mean-squared-variations}
\end{equation}
among all non-constant (i.e.\ non-trivial) functions $f:\mathbb{R}^d\rightarrow \mathbb{R}$ that are properly normalized, where the constant $2$ is introduced for convenience. 

The above reasoning can be extended to the case of vector-valued CV maps. Let us summarize it in the following result. 
\begin{theorem}
Assume that $\omega=(\omega_1,\dots, \omega_m)^\top \in \mathbb{R}^m$ has positive components, where $\omega_1 \ge \omega_2 \ge \dots \ge \omega_m > 0$.
Denote by $|z|_\omega:= (\sum_{i=1}^m \omega_i |z_i|^2)^{1/2}$ the weighted $l^2$-norm of vectors $z=(z_1,\dots, z_m)^\top\in \mathbb{R}^m$. 
The CV map $\xi=(\varphi_1, \dots, \varphi_m)^\top$ defined by the first eigenfunctions $\varphi_1, \dots, \varphi_m$ minimizes the average of squared variations 
\begin{equation}
  \min_{f} \lim_{N\rightarrow +\infty} \frac{1}{2N} \sum_{n=0}^{N-1} |f(X_{n+1}) - f(X_n)|^2_\omega
  \label{mean-squared-variations-vector}
\end{equation}
  among functions $f=(f_1,\dots,f_m)^\top:\mathbb{R}^d\rightarrow
  \mathbb{R}^m$ whose components belong to $\mathcal{H}$ and satisfy 
  \eqref{f-constraints}.
  \label{thm-optimal-cv-map}
\end{theorem}
\begin{proof}
Using the identity \eqref{msv-is-energy} in Lemma~\ref{lemma-energy-and-msv},
  we can derive the objective in \eqref{mean-squared-variations-vector} as
  \begin{align*}
    & \lim_{N\rightarrow +\infty} \frac{1}{2N} \sum_{n=0}^{N-1} |f(X_{n+1}) -
    f(X_n)|^2_\omega \\
    =& \sum_{i=1}^m \omega_i \left(\lim_{N\rightarrow +\infty} \frac{1}{2N} \sum_{n=0}^{N-1}
  |f_i(X_{n+1}) - f_i(X_n)|^2\right) = \sum_{i=1}^m \omega_i \mathcal{E}(f_i)\,.
  \end{align*}
Therefore, Theorem~\ref{thm-variational-form-transfer} implies that the
  minimum of \eqref{mean-squared-variations-vector} under the constraints
  \eqref{f-constraints} is achieved when $f$ is defined by the first $m$ eigenfunctions.
\end{proof}

Note that the norm in \eqref{mean-squared-variations-vector} is an extension of the norm in \eqref{mean-squared-variations} when $k>1$.
The weights $\omega$ are introduced in \eqref{mean-squared-variations-vector} in order to remove the non-uniqueness of
the minimizer of \eqref{mean-squared-variations-vector} due to reorderings of $f_1,f_2,\dots, f_m$.
Condition \eqref{f-constraints} guarantees that $f_1, f_2, \dots, f_m$ are non-constant, normalized, and pairwise orthogonal (therefore linearly independent). 

\subsection{Transition rates} 
\label{subsec-tpt}
In this section, we study the transition of the (discrete-in-time and
continuous-in-space) process $X_n$ between two subsets $A, B \subset
\mathbb{R}^d$ in the framework of
TPT~\cite{tpt_eric2006,towards_tpt2006,tpt2010}. We refer to \cite{ksz2017,Helfmann2020,Helfmann_thesis_2022} for extensions of TPT to different types of Markov chains. 

Let us assume that $A, B \subset \mathbb{R}^d$ are two closed subsets with smooth boundaries $\partial A$ and $\partial B$, respectively, such that
$A\cap B=\emptyset$ and the complement of their union is nonempty, i.e.\ $(A\cup B)^c:=\mathbb{R}^d \setminus (A\cup B)\neq \emptyset$. 
Recall that TPT focuses on the \textit{reactive segments} in the trajectory of $X_n$. 
Specifically, a trajectory segment $X_m, X_{m+1}, \dots, X_{m+j}$, where $m\ge
0$ and $j\ge 1$, is called reactive, if it starts from a state in $A$ and it goes
to $B$ without returning to $A$, or equivalently, if $X_{m}\in A$, $X_{m+j}\in B$, and $X_{i} \in (A\cup B)^c$ for all $m<i<m+j$.
See Figure~\ref{fig-tpt} for illustration. Denote by $M_{N}^R$ the number of reactive
segments that occur within the first $N$ steps of the trajectory (the last reactive segment is allowed to end after the $N$th step).
Then, the transition rate of the process $X_n$ from $A$ to $B$ is defined as 
\begin{equation}
  k_{AB} = \lim_{N\rightarrow +\infty} \frac{M_N^R}{N}\,.
  \label{k-ab-def}
\end{equation}
Let us also recall the (forward) \textit{committor} $q$ associated to the sets $A$ and $B$, which is defined by
\begin{equation}
  q(x) = \mathbb{P}(\text{starting from $x$, $X_n$ enters $B$ before it enters
  $A$}), \quad x \in \mathbb{R}^d\,.
  \label{committor-eqn}
\end{equation}
In mathematical literatures, the committor $q$ is called equilibrium potential, which plays a central role in the study of metastability~\cite{Bovier2004,Bovier2005,Beltran2015}. Calculating the probability in \eqref{committor-eqn} by considering the state in the next step,
it is straightforward to verify that $q$ satisfies the equation (see~\cite[Appendix~A]{tpt_jump} for the proof in the case of Markov jump processes)
\begin{equation}
  \begin{aligned}
    & (\mathcal{T}q)(x) = q(x), \quad \forall~ x \in (A\cup B)^c\,, \\
  & q|_{A} \equiv 0, \quad q|_{B} \equiv 1\,,
  \end{aligned}
  \label{committor-equation}
\end{equation}
where $\mathcal{T}$ is the transfer operator of the process $X_n$. 

In the following, we give alternative expressions for the transition rate $k_{AB}$.
\begin{proposition}
  Let $p(x,y)$ be the transition density of $X_n$ and $\mathcal{E}$ be the Dirichlet energy defined in \eqref{energy}.  Then, the following identities hold.
  \begin{enumerate}
    \item
      $k_{AB}=k_{BA}$.
    \item
      $k_{AB}=\int_{A} \left[\int_{\mathbb{R}^d} p(x,y)
      q(y)\,\dd{y}\right]\mu(\dd{x})=\int_{B} \left[\int_{\mathbb{R}^d} p(x,y) (1-q(y))\,\dd{y}\right] \mu(\dd{x})$.
    \item
      $k_{AB} =
      \mathcal{E}(q)=\frac{1}{2}\int_{\mathbb{R}^d}\left[\int_{\mathbb{R}^d}
      p(x,y) (q(y) - q(x))^2 \dd{y}\right] \mu(\dd{x})$. 
  \end{enumerate}
  \label{prop-kab}
\end{proposition}

The proof of Proposition~\ref{prop-kab} is presented in Appendix~\ref{app-sec-proofs}. 
It is worth mentioning that the proof of Proposition~\ref{prop-kab} does not
rely on reversibility and hence Proposition~\ref{prop-kab} holds for non-reversible ergodic Markov
processes as well. In particular, the third claim of
Proposition~\ref{prop-kab} relates the transition rate $k_{AB}$ to the Dirichlet energy of the committor $q$. 
In the reversible setting (see \eqref{reversibility}), we can further obtain the following variational characterization of the committor.

\begin{proposition}
  Assume that $X_n$ is reversible. Let $q$ be the committor
  \eqref{committor-eqn} associated to the transitions of $X_n$ from $A$ to
  $B$, and let $\mathcal{E}$ denote the Dirichlet energy defined in
  \eqref{energy}. Define the space
  \begin{equation}
    \mathcal{F}_{AB} = \Big\{ f\in \mathcal{H}\,\Big|\, f|_A\equiv 0, ~ f|_B\equiv 1 \Big\}\,.
    \label{f-ab}
  \end{equation}
  We have 
  \begin{equation}
    \mathcal{E}(f) = k_{AB} + \mathcal{E}(f-q)\,, \quad \forall ~ f\in \mathcal{F}_{AB}\,.
    \label{energy-f-q}
  \end{equation}
In particular, $q$ solves the minimization problem
\begin{equation}
  \min_{f\in \mathcal{F}_{AB}} \mathcal{E}(f)\,,
  \label{min-energy-committor}
\end{equation}
  and the minimum value of \eqref{min-energy-committor} is $k_{AB}$.
  \label{prop-committor}
\end{proposition}
\begin{proof}
  It suffices to prove \eqref{energy-f-q}. Let us derive 
  \begin{equation*}
    \begin{aligned}
    \mathcal{E}(f) =& \langle f, (\mathcal{I}-\mathcal{T})f\rangle_\mu\\
      =& \big\langle q + (f - q), (\mathcal{I}-\mathcal{T})\big(q+(f-q)\big)\big\rangle_\mu\\
      =& \langle q, (\mathcal{I}-\mathcal{T})q\rangle_\mu + \langle f-q, (\mathcal{I}-\mathcal{T}) (f-q) \rangle_\mu + 2 \langle q, (\mathcal{I}-\mathcal{T})(f-q)\rangle_\mu   \\
      =& \mathcal{E}(q) + \mathcal{E}(f-q) + 2 \langle (\mathcal{I}-\mathcal{T})q, f-q\rangle_\mu \\
      =& k_{AB} + \mathcal{E}(f-q)\,,
    \end{aligned}
  \end{equation*}
where the first, the third, and the fourth equalities follow from~\eqref{e-f-g-adjoint} in Lemma~\ref{lemma-energy-and-msv} and the
  self-adjointness of $\mathcal{T}$, the last equality follows from the third
  claim of Proposition~\ref{prop-kab}, the fact that $f\equiv q$ on $A\cup B$
  thanks to the boundary condition in \eqref{f-ab} (which is met by both $f$ and $q$), as well as the identity $(\mathcal{I}-\mathcal{T})q=0$ on $(A\cup B)^c$ as stated in the equation \eqref{committor-equation}.
\end{proof}
We note that the variational formulation \eqref{min-energy-committor} is well-known in the literatures on TPT~\cite{tpt2010} and on metastability~\cite{Bovier2004}.  

Similar to Theorem~\ref{thm-optimal-cv-map} which concerns the leading eigenfunctions, we have the following characterization of the committor as a slow variable for the study of transitions from $A$ to $B$.

\begin{theorem}
Assume that $X_n$ is reversible. Given two disjoint subsets $A,B\subset \mathbb{R}^d$, the CV map defined by the committor $q$ minimizes the average of squared variations 
\begin{equation*}
  \min_{f\in \mathcal{F}_{AB}} \lim_{N\rightarrow +\infty} \frac{1}{2N}
  \sum_{n=0}^{N-1} |f(X_{n+1}) - f(X_n)|^2\,,
\end{equation*}
  where $\mathcal{F}_{AB}$ is the space defined in \eqref{f-ab}.
  \label{thm-optimal-cv-map-ab}
\end{theorem}
\begin{proof}
  The conclusion follows directly from Proposition~\ref{prop-committor} and Lemma~\ref{lemma-energy-and-msv}.
\end{proof}

\section{Effective dynamics for a given CV map}
\label{sec-effdyn}
In this section, we study the effective dynamics of $X_n$ for a given CV map. 
We give the definition of the effective dynamics in Section~\ref{subsec-formal-defs} and then study its properties in Section~\ref{subsec-eff-propertities}.
In particular, we show that its transition density can be characterized via
a relative entropy minimization to the transition density of $X_n$ and discuss
the application of this fact in finding CVs (see Proposition~\ref{prop-eff-kl}
and Remark~\ref{rmk-prop-eff-kl}). Unless explicitly stated, in this section we do not require the process $X_n$ to be reversible.

\subsection{Definitions}
\label{subsec-formal-defs}

Before we define the effective dynamics for $X_n$, we need to introduce several quantities associated to the CV map.
Given $1 \le k < d$, let us define the space 
\begin{equation}
  \Xi = \Big\{\xi\in C^3(\mathbb{R}^d,\mathbb{R}^k) \,\Big|\,\mbox{$\xi$ is onto, $\mbox{rank}(\nabla\xi(x)) = k$ at each $x \in \mathbb{R}^d$} \Big\}\,,
  \label{xi-set}
\end{equation}
where $\nabla \xi(x)\in \mathbb{R}^{d\times k}$ denotes the Jacobian matrix of $\xi$.
For $\xi\in\Xi$ and given $z\in \mathbb{R}^k$, the level set
\begin{equation}
  \Sigma_z := \{x\in \mathbb{R}^d\,|\, \xi(x) = z\}
  \label{level-set}
\end{equation}
is a $(d-k)$-dimensional submanifold in $\mathbb{R}^d$
(see~\cite[Remark~1]{LelievreZhang18} and
\cite[Section~1]{multiple-projection-2020}; see Figure~\ref{fig-levelsets} for illustration). We denote by $\sigma_{\Sigma_z}$ the surface measure of $\Sigma_z$
and we define the probability measure $\mu_z$ on $\Sigma_z$ as 
\begin{equation}
  \mu_z(\dd{x}) = p_z(x) \left[\mathrm{det}(\nabla\xi(x)^\top\nabla\xi(x))\right]^{-\frac{1}{2}} \sigma_{\Sigma_z}(\dd{x})\,, \quad p_z(x) =
  \frac{\pi(x)}{\widetilde{\pi}(z)}  \,,
  \label{mu-z}
\end{equation}
where $\widetilde{\pi}(z)$ is the normalizing constant defined by 
\begin{equation}
  \widetilde{\pi}(z)= \int_{\Sigma_z} \pi(x)
  [\mathrm{det}(\nabla\xi(x)^\top\nabla\xi(x))]^{-\frac{1}{2}} \,\sigma_{\Sigma_z}(\dd{x}) \,.
  \label{q-eff}
\end{equation}
A determinant term is introduced in \eqref{mu-z} and \eqref{q-eff}, so that by co-area formula~(see \cite[Theorem 3.11]{evansmeasure} and \cite[Lemma~3.2]{lelievre2010free}) we have, for all integrable functions $f: \mathbb{R}^d\rightarrow \mathbb{R}$,
\begin{equation}
  \int_{\mathbb{R}^d} f(x) \mu(\dd{x}) = \int_{\mathbb{R}^k}
  \Big(\int_{\Sigma_z} f(x)\, \mu_z(\dd{x})\Big) \widetilde{\mu}(\dd{z})\,,
  \label{co-area}
\end{equation}
where $\widetilde{\mu}$ is the pushforward measure of $\mu$ by the map $\xi$, given by 
\begin{equation}
  \widetilde{\mu}(\dd{z}) = \widetilde{\pi}(z) \dd{z}\,,\quad z \in \mathbb{R}^k\,.
  \label{mu-eff}
\end{equation}
Note that \eqref{co-area} is identical to the identity in the disintegration theorem for probability measures~\cite[Theorem~5.3.1]{ambrosio2005gradient}.
The measure $\mu_z$ is often termed as the conditional probability measure, since by \eqref{co-area} and the law of total expectation it holds
\begin{equation*}
  \mathbb{E}_{\mu_z}(f) = \int_{\Sigma_z} f(x) \,\mu_z(\dd{x}) = \mathbb{E}_{\mu}(f(x)\,|\,\xi(x)=z)\,,
\end{equation*}
for all integrable functions $f$. 
In particular, when $\xi$ is a linear map, the level set $\Sigma_z$ is a linear subspace and the expressions of quantities introduced above can be largely simplified~\cite{LEGOLL2017-pathwise,LegollLelievreSharma18}.

To define the effective dynamics for the process $X_n$, we borrow the idea
from~\cite{effective_dynamics} for the effective dynamics of SDEs.
Specifically, notice that $\xi(X_n)$ is in general a non-Markovian process in
$\mathbb{R}^k$, since the probability that $\xi(X_{n+1})\in\widetilde{D}$ at $(n+1)$th step for a subset $\widetilde{D}\subset \mathbb{R}^k$ depends on the
full state $X_n=x$ at $n$th step rather than $\xi(X_n)=z$ only, i.e.
\begin{equation}
  \begin{aligned}
    \mathbb{P}(\xi(X_{n+1})\in \widetilde{D}|X_n=x)=& \int_{\mathbb{R}^d} p(x, y) \mathbbm{1}_{\widetilde{D}}(\xi(y)) \,\dd{y}  \\
    =& \int_{\widetilde{D}} \left(\int_{\Sigma_{w}} p(x, y) 
[\mathrm{det}(\nabla\xi(y)^\top\nabla\xi(y))]^{-\frac{1}{2}} \sigma_{\Sigma_z}(\dd{y})\right) \,\dd{w}\,,
  \end{aligned}
  \label{p-xi-non-closed}
\end{equation}
where $\mathbbm{1}_{\widetilde{D}}$ denotes the indicator function of the set
$\widetilde{D}$ and the expression on the second line follows from co-area formula~\cite[Lemma~3.2]{lelievre2010free}.
Following the idea in~\cite{effective_dynamics}, one way to define a transition
density in $\mathbb{R}^k$ is to replace the expression on the right hand side
of \eqref{p-xi-non-closed} by its average over the states $X_n=x$ on $\Sigma_z$ with respect to the conditional measure $\mu_z$ in \eqref{mu-z}.
Specifically, given the CV map $\xi$, we define the effective dynamics of $X_n$ as follows.
\begin{definition}[Effective dynamics]
  The effective dynamics of $X_n$ associated to the CV map $\xi:
  \mathbb{R}^d\rightarrow \mathbb{R}^k$ is defined as the Markov process $Z_n$ in $\mathbb{R}^k$ with the transition density 
$\widetilde{p}(z,w)$, where
\begin{equation}
  \begin{aligned}
    \widetilde{p}(z,w)
=& \int_{\Sigma_{z}}\Big[ \int_{\Sigma_{w}}
    p(x, y) [\mathrm{det}(\nabla\xi(y)^\top\nabla\xi(y))]^{-\frac{1}{2}}
    \sigma_{\Sigma_{w}}(\dd{y}) \Big]\,\mu_z(\dd{x})\\
    =& \widetilde{\pi}(w) \int_{\Sigma_{z}}\left(\int_{\Sigma_{w}}
    \frac{p(x,y)}{\pi(y)} \mu_w(\dd{y}) \right)\,\mu_z(\dd{x})\,, \quad z,w \in \mathbb{R}^k\,. 
  \end{aligned}
  \label{p-z}
\end{equation}
  \label{def-effdyn}
\end{definition}
Note that the second equality in \eqref{p-z} follows from the expression in \eqref{mu-z}.
Using \eqref{co-area} and the fact that $p(x,\cdot)$ is a probability
density in $\mathbb{R}^d$, we can verify that $\widetilde{p}(z,\cdot)$ in
\eqref{p-z} indeed defines a probability density for $z\in \mathbb{R}^k$, i.e.\ $\int_{\mathbb{R}^k}
\widetilde{p}(z,w)\,\dd{w}=1$. As we show in the next subsection,
the invariant distribution of $Z_n$ is related to the free energy associated to
$\xi$~\cite[Section~3.2.1]{lelievre2010free} and its transition density
$\widetilde{p}$ gives the best approximation of $p$ in the sense of relative
entropy minimization~\cite[Section 2.3]{Cover2006}. As discussed above, the effective dynamics defined in Definition~\ref{def-effdyn} can be viewed as
the counterpart to the effective dynamics~\cite{effective_dynamics} for discrete-in-time and continuous-in-space Markov processes.
It also shares similarities with the effective dynamics of Markov chains in
previous works~\cite{zhang2016asymptoticanalysismultiscalemarkov,sharma-cg-markovchain}
and with the asymptotic Markov chains considered in the study of metastability~\cite{Beltran2010,Beltran2015,landim-resolvent-metastability}.

\begin{figure}[t!]
  \centering
  \includegraphics[width=0.5\textwidth]{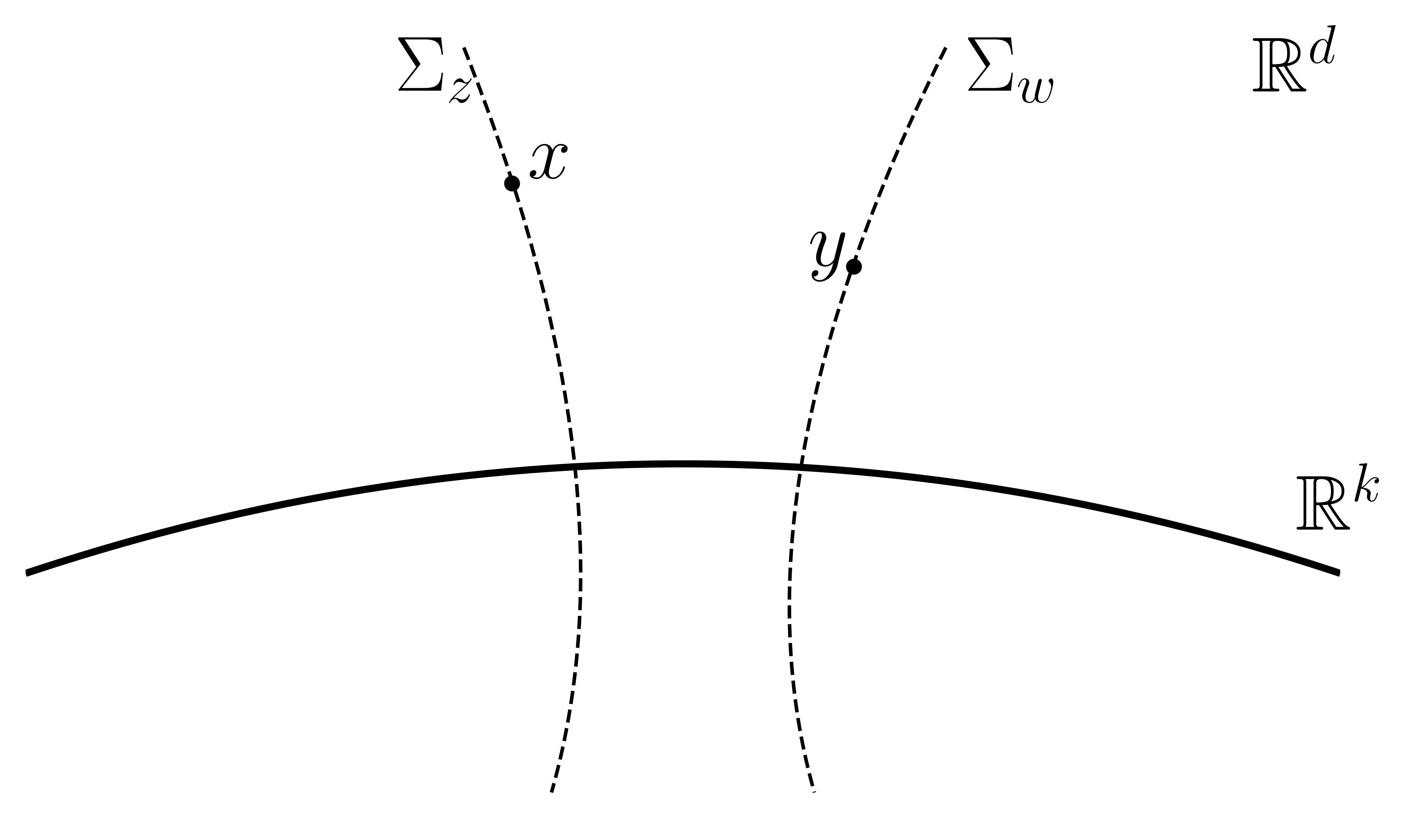}
  \caption{Illustration of level sets of the CV map $\xi: \mathbb{R}^d\rightarrow \mathbb{R}^k$ for values $z,w\in \mathbb{R}^k$. \label{fig-levelsets}}
\end{figure}

\subsection{Properties of effective dynamics}
\label{subsec-eff-propertities}

In this section, we study properties of the effective dynamics $Z_n$ in Definition~\ref{def-effdyn} whose transition density $\widetilde{p}$ is defined in \eqref{p-z}. Let us first identify its invariant probability measure.
\begin{proposition}
  The probability measure $\widetilde{\mu}$ defined in \eqref{mu-eff}  is an invariant measure of $Z_n$. 
  \label{prop-inv}
\end{proposition}
\begin{proof}
  To show the invariance of $\widetilde{\mu}$, let us compute, for $w \in \mathbb{R}^k$,
  \begin{align*}
    \left(\int_{\mathbb{R}^k} \widetilde{p}(z,w)\widetilde{\mu}(\dd{z})\right)\dd{w}\,  =& 
\widetilde{\pi}(w)\dd{w} 
\int_{\mathbb{R}^k} \int_{\Sigma_{z}}\left(\int_{\Sigma_{w}}
    \frac{p(x,y)}{\pi(y)} \mu_w(\dd{y}) \right)\,\mu_z(\dd{x})\widetilde{\mu}(\dd{z})\\
      =& \widetilde{\pi}(w)\dd{w} 
\int_{\mathbb{R}^d} \left(\int_{\Sigma_{w}} \frac{p(x,y)}{\pi(y)} \mu_w(\dd{y}) \right)\,\mu(\dd{x})\\
      =& \widetilde{\pi}(w)\dd{w} 
 \left(\int_{\Sigma_{w}} \frac{\int_{\mathbb{R}^d}p(x,y)\mu(\dd{x})}{\pi(y)}\,\mu_w(\dd{y}) \right)\\
      =& \widetilde{\pi}(w)\dd{w} = \widetilde{\mu}(\dd{w})\,,
    \end{align*}
  where the first equality follows from the second equality in \eqref{p-z}, the second equality
  follows from \eqref{co-area}, the third equality follows from an exchange of
  integrals, and the fourth equality follows from the invariance of $\mu$ in \eqref{eqn-p-pi}. 
\end{proof}

From now on, we will assume that the effective dynamics $Z_n$ is ergodic and $\widetilde{\mu}$ is its unique invariant probability measure.
Notice that, when $X_n$ comes from Brownian dynamics \eqref{sde-overdamped} with invariant density $\pi(x) = \frac{1}{Z} \mathrm{e}^{-\beta V(x)}$
(Example~\ref{example-concrete}), 
the free energy associated to the CV map $\xi$, denoted by the map $F:\mathbb{R}^k\rightarrow \mathbb{R}$, is related to the normalizing constant
in \eqref{q-eff} by $F(z)=-\beta^{-1}\ln \widetilde{\pi}(z)$~\cite[Section~3.2.1]{lelievre2010free}, where $\beta>0$ is a constant related to the system's temperature.
Therefore, Proposition~\ref{prop-inv} states that $Z_n$ as a surrogate model has the desired invariant probability measure whose density is $\widetilde{\pi}(z)=\exp(-\beta F(z))$, for $z\in \mathbb{R}^k$. 

In the following, we present a variational characterization of the transition
density~$\widetilde{p}$ of $Z_n$ by showing that $\widetilde{p}$, together
with the conditional measure $\mu_z$, gives the best approximation of $p$
within certain class of transition densities in terms of relative entropy minimization.
To this end, note that the transition density $p$ can be written as 
\begin{equation}
p(x,y)=p(w|x)\,p(y|w,x), 
  \label{p-as-product-of-conditional-density}
\end{equation}
  where $p(w|x)$ denotes the probability density of $w=\xi(y)\in
  \mathbb{R}^k$ given the value $x\in \mathbb{R}^d$, and $p(y|w,x)$ is
  the probability density of $y\in \Sigma_{w}$ given the values $w$ and $x$ (see Figure~\ref{fig-levelsets} for illustration). 
  Let us define
\begin{align}
  \begin{split}
    \widetilde{\mathcal{S}} =& \left\{\widetilde{f} \in
    C(\mathbb{R}^k,\mathbb{R}^+)\,\middle|\, \int_{\mathbb{R}^k}
    \widetilde{f}(z)\,\dd{z} = 1 \right\}\,, \\
    \widetilde{\mathcal{G}} =& \left\{\widetilde{g} \in C(\mathbb{R}^k\times \mathbb{R}^k, \mathbb{R}^+)\,\middle|\, \int_{\mathbb{R}^k} \widetilde{g}(z,w)\,\dd{w} =
    1\,,\forall z\in \mathbb{R}^k\right\}\,, \\
    \mathcal{G}_{\xi,\Sigma} =& \left\{(g_z)_{z\in \mathbb{R}^k} \,\middle|\, g_z\in C(\Sigma_z, \mathbb{R}^+), 
    \vphantom{\int_{\Sigma_{z}}} \right. \\
      & \hspace{1.6cm} \left. \int_{\Sigma_{z}} g_z(x)\,\left[\mathrm{det}(\nabla\xi(x)^\top\nabla\xi(x))\right]^{-\frac{1}{2}} \sigma_{\Sigma_z}(\dd{x}) = 1\,,\forall z\in \mathbb{R}^k\right\}\,,
  \end{split}
  \label{spaces-of-densities}
\end{align}
which are spaces of probability densities in $\mathbb{R}^k$, transition densities in $\mathbb{R}^k$, and 
  probability densities on level sets, respectively. 
  One way to find a Markov process on $\mathbb{R}^k$ that mimics the dynamics of
  $\xi(X_n)$ is to consider the approximation of $p$ within
  the set $\Theta_\xi$ of transition probability densities defined as:
\begin{align}
  \begin{split}
  \Theta_\xi = \bigg\{ g \in C(\mathbb{R}^d\times \mathbb{R}^d, \mathbb{R}^+)\,\bigg|\,&
    g(x,y)=\widetilde{g}(\xi(x),\xi(y))\,g_{\xi(y)}(y),\, \mbox{for}\, x,y \in
    \mathbb{R}^d\,, \\
    & \mbox{where} ~ \widetilde{g}\in \widetilde{\mathcal{G}}, ~ (g_{z})_{z\in \mathbb{R}^k} \in \mathcal{G}_{\xi,\Sigma} \bigg\}\,.
  \end{split}
  \label{theta-density-set}
\end{align}
In other words, we approximate the two terms on the right hand side of
\eqref{p-as-product-of-conditional-density} by assuming that the first term is
only conditioned on $z=\xi(x)$ rather than $x$ and the second term is only conditioned on $w$. 
Note that the definition of $\mathcal{G}_{\xi,\Sigma}$ in
\eqref{spaces-of-densities} makes sure that, by
co-area formula, elements in \eqref{theta-density-set} are transition densities.

The following result states that, among the probability densities in
\eqref{theta-density-set}, the one which minimizes the relative entropy (or, KL-divergence) to $p$ is given by the transition density
$\widetilde{p}$ \eqref{p-z} of the effective dynamics $Z_n$ in
Definition~\ref{def-effdyn} and the conditional expectation $\mu_z$ \eqref{mu-z}. 

\begin{proposition}
For $x\in \mathbb{R}^d$, denote by $D_{\mathrm{KL}}(p(x,\cdot)\|g(x,\cdot))$ the KL-divergence from the density $g(x,\cdot)$ to the density $p(x,\cdot)$. 
  Then the minimizer 
  \begin{equation}
    g^{\mathrm{opt}}=\arg\min_{g\in \Theta_\xi} \mathbb{E}_{x\sim \mu}\Big(D_{\mathrm{KL}}\big(p(x,\cdot)\,\|\,g(x,\cdot)\big)\Big)
    \label{kl-minimization}
  \end{equation}
  is given by $g^{\mathrm{opt}}(x,y)=\widetilde{p}(\xi(x),\xi(y))\,p_{\xi(y)}(y)$, for
  $x,y\in \mathbb{R}^d$, where $\widetilde{p}$ is the transition density
  defined in~\eqref{p-z} and $p_z$ is defined in \eqref{mu-z} (in the definition of $\mu_z$). 
  \label{prop-eff-kl}
\end{proposition}
\begin{proof}
  For any $g\in \Theta_\xi$ with the expression 
   \begin{equation}
     g(x,y)=\widetilde{g}(\xi(x),\xi(y))\,g_{\xi(y)}(y)\,, 
     \label{g-exp-in-theta}
   \end{equation}
   we show the following identities hold.
   \begin{align}
    &\mathbb{E}_{x\sim \mu}\Big(D_{\mathrm{KL}}\big(p(x,\cdot)\,\|\,g(x,\cdot)\big)\Big)
      - \int_{\mathbb{R}^d}\bigg(\int_{\mathbb{R}^d} p(x,y)\ln p(x,y) \,\dd{y}\bigg) \mu(\dd{x}) \notag \\
    =& 
      -\int_{\mathbb{R}^d}\bigg(\int_{\mathbb{R}^d} p(x,y)\ln \widetilde{g}(\xi(x),\xi(y)) \dd{y}\bigg) \mu(\dd{x}) - \int_{\mathbb{R}^d}\ln g_{\xi(x)}(x) \mu(\dd{x})\notag \\
    =& -\int_{\mathbb{R}^d}\bigg(\int_{\mathbb{R}^d} p(x,y)\ln \widetilde{g}(\xi(x),\xi(y)) \dd{y}\bigg) \mu(\dd{x})
     + \int_{\mathbb{R}^k} D_{\mathrm{KL}}(p_{z}\,\|\,g_{z}) \widetilde{\mu}(\dd{z})\notag\\
      & - \int_{\mathbb{R}^d} \ln p_{\xi(x)}(x) \mu(\dd{x}) \label{kl-derive-3}\\
    =&
      \int_{\mathbb{R}^k}\bigg(D_{\mathrm{KL}}\big(\widetilde{p}(z,\cdot)\,\|\,\widetilde{g}(z,\cdot)\big)
      + D_{\mathrm{KL}}(p_{z}\,\|\,g_{z}) \bigg) \widetilde{\mu}(\dd{z})\notag\\
    &- \int_{\mathbb{R}^k}\bigg(\int_{\mathbb{R}^k}
      (\widetilde{p}\ln\widetilde{p})(z,w)\,
      \dd{w}\bigg)\widetilde{\mu}(\dd{z}) - \int_{\mathbb{R}^d} \ln
      p_{\xi(x)}(x) \mu(\dd{x})\,.\notag
  \end{align}
In the above, $D_{\mathrm{KL}}(p_{z}\,\|\,g_{z})$ denotes the KL divergence of $g_{z}$ from $p_z$, both of which are densities on $\Sigma_{z}$ 
  with respect to $\left[\mathrm{det}(\nabla\xi(x)^\top\nabla\xi(x))\right]^{-\frac{1}{2}} \sigma_{\Sigma_z}(\dd{x})$.

  We start by calculating the first term on the first line of \eqref{kl-derive-3}. By definition of KL-divergence~\cite[Section
  2.3]{Cover2006}, we have, for $g\in \Theta_\xi$ with expression in \eqref{g-exp-in-theta}, 
\begin{align}
    &
    \mathbb{E}_{x\sim\mu}\Big(D_{\mathrm{KL}}\big(p(x,\cdot)\,\|\,g(x,\cdot)\big)\Big) \label{kl-derive-1}\\
    =& \int_{\mathbb{R}^d}\bigg(\int_{\mathbb{R}^d}
    p(x,y)\ln\frac{p(x,y)}{g(x,y)} \,\dd{y}\bigg) \mu(\dd{x})\notag \\
    =& \int_{\mathbb{R}^d}\bigg(\int_{\mathbb{R}^d} p(x,y)\ln p(x,y)
    \,\dd{y}\bigg) \mu(\dd{x}) -
    \int_{\mathbb{R}^d}\bigg(\int_{\mathbb{R}^d} p(x,y)\ln g(x,y)\,\dd{y}\bigg) \mu(\dd{x})\,,\notag
  \end{align}
  where the first term on the last line is independent of $g$. For the second term above, we can derive 
\begin{align}
    &-\int_{\mathbb{R}^d}\bigg(\int_{\mathbb{R}^d} p(x,y)\ln g(x,y)\,\dd{y}\bigg) \mu(\dd{x}) \label{kl-derive-2}\\
    =\,& -\int_{\mathbb{R}^d}\bigg(\int_{\mathbb{R}^d} p(x,y)\left(\ln
    \widetilde{g}(\xi(x),\xi(y)) +\ln g_{\xi(y)}(y)\right)\,\dd{y}\bigg) \mu(\dd{x}) \notag \\
    =\,& -\int_{\mathbb{R}^d}\bigg(\int_{\mathbb{R}^d} p(x,y)\ln
    \widetilde{g}(\xi(x),\xi(y)) \dd{y}\bigg) \mu(\dd{x}) -
    \int_{\mathbb{R}^d}\ln g_{\xi(x)}(x) \mu(\dd{x}) \notag \\
    =:& I_1 + I_2\,,\notag
\end{align}
where the first equality follows from the expression of $g$ in \eqref{g-exp-in-theta}, and the second equality follows from the invariance
  of $\mu$ in \eqref{eqn-p-pi}.  We continue to compute $I_1$ and $I_2$ above.
For $I_1$, we have 
\begin{align}
    I_1=& -\int_{\mathbb{R}^d}\bigg(\int_{\mathbb{R}^d} p(x,y)\ln \widetilde{g}(\xi(x),\xi(y)) \dd{y}\bigg) \mu(\dd{x}) \label{kl-derive-2-i1} \\
    =& - \int_{\mathbb{R}^k}\bigg(\int_{\mathbb{R}^k} \widetilde{p}(z,w)
    \ln \widetilde{g}(z,w)\, \dd{w}\bigg) \widetilde{\mu}(\dd{z}) \notag \\
    =&
    \int_{\mathbb{R}^k}D_{\mathrm{KL}}\big(\widetilde{p}(z,\cdot)\,\|\,\widetilde{g}(z,\cdot)\big) \widetilde{\mu}(\dd{z}) 
     - \int_{\mathbb{R}^k}\bigg(\int_{\mathbb{R}^k}
     (\widetilde{p}\ln\widetilde{p})(z,w)\, \dd{w}\bigg) \widetilde{\mu}(\dd{z}) \,.  \notag
\end{align}
where the first equality follows from the co-area formula, the identity \eqref{co-area} and the definition of $\widetilde{p}$ in
  \eqref{p-z}, and the second equality follows from the definition of KL-divergence.

Similarly, for $I_2$, we have 
\begin{align}
  I_2 =& - \int_{\mathbb{R}^d}\ln g_{\xi(x)}(x) \mu(\dd{x}) \label{kl-derive-2-i2} \\
  =& - \int_{\mathbb{R}^k} \left(\int_{\Sigma_z} \ln g_{z}(x) \mu_z(\dd{x})\right)\widetilde{\mu}(\dd{z}) \notag \\
  =& - \int_{\mathbb{R}^k} \left(\int_{\Sigma_z} \ln g_{z}(x) 
p_z(x) \left[\mathrm{det}(\nabla\xi(x)^\top\nabla\xi(x))\right]^{-\frac{1}{2}} \sigma_{\Sigma_z}(\dd{x}) \right)\widetilde{\mu}(\dd{z}) \notag \\
    =& \int_{\mathbb{R}^k}D_{\mathrm{KL}}(p_{z}\,\|\,g_{z}) \widetilde{\mu}(\dd{z})
    - \int_{\mathbb{R}^k}\bigg( \int_{\Sigma_{z}} \ln p_z(x) \mu_{z}(\dd{x})
    \bigg) \widetilde{\mu}(\dd{z}) \notag \\
    =& \int_{\mathbb{R}^k}D_{\mathrm{KL}}(p_{z}\,\|\,g_{z}) \widetilde{\mu}(\dd{z}) - \int_{\mathbb{R}^d}\ln p_{\xi(x)}(x) \mu(\dd{x}) \,.\notag
\end{align}
Hence, we obtain the equalities in \eqref{kl-derive-3} by combining \eqref{kl-derive-1}--\eqref{kl-derive-2-i2}.

  Since the KL-divergence is nonnegative and it equals to zero if and only if the two densities are identical, we
  conclude from the last equality in \eqref{kl-derive-3} that the objective in
  \eqref{kl-minimization} is minimized when $g$ is given in
  \eqref{g-exp-in-theta} with $\widetilde{g}=\widetilde{p}$ and $g_z=p_z$ for $z\in \mathbb{R}^k$.
\end{proof}

Proposition~\ref{prop-eff-kl} shows that, given the CV map $\xi$, the
effective dynamics $Z_n$ in Definition~\ref{def-effdyn} with the transition density $\widetilde{p}$ in \eqref{p-z} is the optimal
surrogate model in the sense of relative entropy minimization \eqref{theta-density-set}--\eqref{kl-minimization}.
But, apparently the ability of $Z_n$ in approximating/representing the true
dynamics $X_n$ depends on the choice of $\xi$. In Section~\ref{sec-error-estimates} we will study the approximation quality of
$Z_n$ in terms of estimating both the spectrum and transition rates of the true
dynamics $X_n$. We note that conditions on the choice of $\xi$ which gives a good effective dynamics have been investigated in the \textit{transition manifold} framework~\cite{Bittracher2018,optimal-rc-bittracher}.

In the following result, we transform Proposition~\ref{prop-eff-kl} into
variational formulations that allow for optimizing the CV map $\xi$. 

\begin{theorem}[Optimizing CV maps by relative entropy
  minimization]\label{thm:eff-kl} The CV map $\xi$ with the least
  KL-divergence between the full dynamics and the effective dynamics solves
  the following equivalent optimization tasks:
  \begin{align}
    &  \min_{\xi\in \Xi} \min_{g\in \Theta_\xi} \mathbb{E}_{x\sim
    \mu}\Big(D_{\mathrm{KL}}\big(p(x,\cdot)\,\|\,g(x,\cdot)\big)\Big) \label{optimal-xi-kl-minimization}\\
     \Longleftrightarrow &\min_{\substack{\xi\in \Xi, \widetilde{g}\in
     \widetilde{\mathcal{G}},\\ (g_z)\in \mathcal{G}_{\xi,\Sigma}}}
     \bigg[-\int_{\mathbb{R}^d} \Big(\int_{\mathbb{R}^d} p(x,y)\ln \widetilde{g}(\xi(x),\xi(y))\,\dd{y}\Big)
      \mu(\dd{x}) -  \int_{\mathbb{R}^d} \ln g_{\xi(x)}(x) \mu(\dd{x})\bigg] \label{optimal-xi-kl-minimization-1}\\
     \Longleftrightarrow &\min_{\xi\in \Xi, \widetilde{g}\in \widetilde{\mathcal{G}}} \bigg[-\int_{\mathbb{R}^d}
    \Big(\int_{\mathbb{R}^d} p(x,y)\ln \widetilde{g}(\xi(x),\xi(y))\,\dd{y}\Big)
      \mu(\dd{x}) +  \max_{\widetilde{f}\in \widetilde{\mathcal{S}}} \int_{\mathbb{R}^d} \ln \widetilde{f}(\xi(x)) \mu(\dd{x})  \bigg]\,, \label{optimal-xi-kl-minimization-2}
  \end{align}
  where $\Xi$ is the space of CV maps in \eqref{xi-set}, $\Theta_\xi$ is the space of transition densities defined in \eqref{theta-density-set}, and
  $\widetilde{\mathcal{G}}$, $\widetilde{\mathcal{S}}$, $\mathcal{G}_{\xi,\Sigma}$
  are the spaces of densities defined in \eqref{spaces-of-densities}.
\end{theorem}
\begin{proof}
  We apply the equalities in \eqref{kl-derive-3}.
  The equivalence between \eqref{optimal-xi-kl-minimization} and \eqref{optimal-xi-kl-minimization-1}
  follows directly from the first equality in \eqref{kl-derive-3}, noticing
  that the second term on the first line of \eqref{kl-derive-3} is independent of both $\xi$ and $g$. 
Next, we show the equivalence between \eqref{optimal-xi-kl-minimization} and \eqref{optimal-xi-kl-minimization-2}.
  Note that the second equality in \eqref{kl-derive-3} (since the
  divergence $D_{KL}(p_z\|g_z)$ is nonnegative and it vanishes when
  $g_z=p_z$) implies that \eqref{optimal-xi-kl-minimization} is equivalent to 
  \begin{equation}
    \min_{\xi\in \Xi, \widetilde{g}\in \widetilde{\mathcal{G}}} \bigg[-\int_{\mathbb{R}^d}\Big(\int_{\mathbb{R}^d}
    p(x,y)\ln \widetilde{g}(\xi(x),\xi(y))\, \dd{y}\Big)\,\mu(\dd{x}) - \int_{\mathbb{R}^d} \ln p_{\xi(x)}(x) \mu(\dd{x}) \bigg]\,.
    \label{optimal-xi-kl-minimization-explicit-tmp}
  \end{equation}
Using the expression of the density $p_z$ in \eqref{mu-z}, we obtain
  \begin{align}
    \begin{split}
    & -\int_{\mathbb{R}^d} \ln p_{\xi(x)}(x) \mu(\dd{x}) + \int_{\mathbb{R}^d}  \ln \pi(x) \mu(\dd{x}) \\
      = & \int_{\mathbb{R}^d}  \ln \widetilde{\pi}(\xi(x)) \mu(\dd{x}) \\
      =& \max_{\widetilde{f}\in \widetilde{\mathcal{S}}} \int_{\mathbb{R}^d} \ln \widetilde{f}(\xi(x)) \mu(\dd{x})
    \label{optimal-xi-kl-minimization-explicit-tmp-1}
    \end{split}
  \end{align}
  where $\widetilde{\pi}$ is the probability density on $\mathbb{R}^k$ defined in
  \eqref{q-eff}, and the second equality can be verified using co-area formula
  and the fact that $D_{KL}(\widetilde{\pi}\|\widetilde{f})$ is nonnegative for any
  probability density $\widetilde{f} \in \widetilde{\mathcal{S}}$.

The equivalence between \eqref{optimal-xi-kl-minimization} and
  \eqref{optimal-xi-kl-minimization-2} follows after we substituting  
    \eqref{optimal-xi-kl-minimization-explicit-tmp-1} into
    \eqref{optimal-xi-kl-minimization-explicit-tmp} and noticing that the
    second term on the first line of \eqref{optimal-xi-kl-minimization-explicit-tmp-1} is independent of both
    $\xi$ and $\widetilde{g}$.
\end{proof}

\begin{remark}[Algorithmic realization]
Theorem~\ref{thm:eff-kl} provides two possible data-driven numerical approaches for
  jointly learning $\xi$ and the transition density $\widetilde{p}$ of the effective dynamics.
  (see \cite[Section~4]{optimal-rc-bittracher} for alternative objective loss functions involving weighted $L^1$ norm of densities). 
  \begin{enumerate}
    \item
  The first approach is to solve \eqref{optimal-xi-kl-minimization} (or \eqref{optimal-xi-kl-minimization-1}). It boils down to minimizing the negative log-likelihood with trajectory data $X_0, X_1,\dots, X_N$:
  \begin{equation}
\min_{\xi\in \Xi} \min_{g \in \Theta_\xi} \Big[-\frac{1}{N} \sum_{n=0}^{N-1} \ln g(X_n, X_{n+1}) \Big]\,.
    \label{optimal-xi-likelihood}
  \end{equation}
  This is close to the loss function proposed in \cite{hao-rc-flow} for learning CV maps based on normalizing flows~\cite{KobyzevPAMI2020}.
\item
  The second approach is to solve \eqref{optimal-xi-kl-minimization-2}. With trajectory data, it amounts to solving 
  \begin{align}
    \begin{split}
      \min_{\xi\in \Xi} \Big[& -\max_{\widetilde{g} \in
      \widetilde{\mathcal{G}}} \frac{1}{N} \sum_{n=0}^{N-1} \ln \widetilde{g}(\xi(X_n),
      \xi(X_{n+1})) +  \max_{\widetilde{f} \in \widetilde{\mathcal{S}}}
      \frac{1}{N}\sum_{n=0}^{N-1} \ln \widetilde{f}(\xi(X_n)) \Big]\,.
    \end{split}
    \label{optimal-xi-likelihood-2nd}
  \end{align}
  \end{enumerate}
In \eqref{optimal-xi-likelihood}, learning $g$ requires to learn the densities
  $(g_{z})_{z\in \mathbb{R}^k}$ on (high-dimensional) level sets. 
  In contrast, both $\widetilde{g}$ and $\widetilde{f}$ in
  \eqref{optimal-xi-likelihood-2nd} are functions in the low-dimensional space $\mathbb{R}^k$.
Hence, we expect the second approach to have less complexity in comparison with the first approach.
  \label{rmk-prop-eff-kl}
\end{remark}

Next, we study the transfer operator of the effective dynamics. In analogy to the transfer operator $\mathcal{T}$ of $X_n$ (see \eqref{trans-operator}), we define the transfer operator associated to $Z_n$ by 
\begin{equation}
  (\widetilde{\mathcal{T}}\widetilde{f}\,)(z) = \int_{\mathbb{R}^k}
  \widetilde{p}(z,w) \widetilde{f}(w)\, \dd{w}\,, \quad z\in \mathbb{R}^k
  \label{trans-operator-eff}
\end{equation}
for functions in the Hilbert space 
\begin{equation*}
  \widetilde{\mathcal{H}} = L^2_{\widetilde{\mu}}(\mathbb{R}^k):=
  \Big\{\widetilde{f}\, \Big|\, \widetilde{f}: \mathbb{R}^k\rightarrow
  \mathbb{R},\, \int_{\mathbb{R}^k} \widetilde{f}^{\,2}(z)\, \widetilde{\mu}(\dd{z}) < +\infty\Big\}\,,
  \end{equation*}
  which is endowed with the weighted inner product (the corresponding norm is denoted by $\|\cdot\|_{\widetilde{\mu}}$)
  \begin{equation}
    \langle \widetilde{f}, \widetilde{h}\,\rangle_{\widetilde{\mu}} =
    \int_{\mathbb{R}^k} \widetilde{f}(z) \widetilde{h}(z) \widetilde{\mu}(\dd{z}) \,,
    \quad \widetilde{f}, \widetilde{h} \in \widetilde{\mathcal{H}}\,.
    \label{inner-product-eff}
  \end{equation}

The following lemma reveals the relations between the transfer operators $\widetilde{\mathcal{T}}$ and $\mathcal{T}$. 
\begin{lemma}
  Let $\mathcal{T}$ be the transfer operator in \eqref{trans-operator} associated to $X_n$ and $\widetilde{\mathcal{T}}$ be the transfer operator in \eqref{trans-operator-eff} associated to $Z_n$. Then, the following relations hold.
  \begin{enumerate}
    \item
      For any $\widetilde{f}\in \widetilde{\mathcal{H}}$, we have $(\widetilde{\mathcal{T}}\,\widetilde{f})(z) =
      \mathbb{E}_{\mu_z}(\mathcal{T}f)$ for all $z\in \mathbb{R}^k$, where $f=\widetilde{f}\circ \xi$.
    \item
      For any two functions $\widetilde{f}, \widetilde{h}\in \widetilde{\mathcal{H}}$, we have $\langle \widetilde{\mathcal{T}}
      \widetilde{f}, \widetilde{h} \rangle_{\widetilde{\mu}}=\langle \mathcal{T}f, h\rangle_{\mu}$, where $f=\widetilde{f}\circ \xi$ and $h=\widetilde{h}\circ \xi$.
  \end{enumerate}
  \label{lemma-tran-tran-eff}
\end{lemma}
\begin{proof}
  Concerning the first claim, given $\widetilde{f}\in \widetilde{\mathcal{H}}$, let us compute 
  \begin{equation*}
    \begin{aligned}
      (\widetilde{\mathcal{T}}\,\widetilde{f})(z) =& \int_{\mathbb{R}^k} \widetilde{p}(z,w) \widetilde{f}(w)\, \dd{w} \\
      =&\int_{\mathbb{R}^k}
\left[ \widetilde{\pi}(w) \int_{\Sigma_{z}}\left(\int_{\Sigma_{w}} \frac{p(x,y)}{\pi(y)} \mu_w(\dd{y}) \right)\,\mu_z(\dd{x})
       \right] \widetilde{f}(w)\, \dd{w} \\
      =&
\int_{\Sigma_{z}}
\left[ \int_{\mathbb{R}^k}  \left(\int_{\Sigma_{w}} \frac{p(x,y)}{\pi(y)} \mu_w(\dd{y}) \right)\,
        \widetilde{f}(w)\,\widetilde{\mu}(\dd{w})\right]\mu_z(\dd{x}) \\
      =&
\int_{\Sigma_{z}}
      \left( \int_{\mathbb{R}^d}   \frac{p(x,y)}{\pi(y)} \widetilde{f}(\xi(y))\mu(\dd{y}) \right)\mu_z(\dd{x}) \\
      =&
\int_{\Sigma_{z}} \left( \int_{\mathbb{R}^d} p(x,y) \widetilde{f}(\xi(y))\dd{y} \right)\mu_z(\dd{x}) \\
      =&\mathbb{E}_{\mu_z}(\mathcal{T}f)\,,
    \end{aligned}
  \end{equation*}
  where $f=\widetilde{f}\circ\xi$, and we have used
  \eqref{trans-operator-eff}, the second equality in \eqref{p-z}, as well as \eqref{co-area} to derive the above equalities. 
  Concerning the second claim, using \eqref{mu-z}, \eqref{inner-product-eff},
  the first claim, and the identity~\eqref{co-area}, we can compute
  \begin{equation*}
    \begin{aligned}
      \langle \widetilde{\mathcal{T}}\,\widetilde{f},
      \widetilde{h}\rangle_{\widetilde{\mu}} =& \int_{\mathbb{R}^k}
      \big(\mathbb{E}_{\mu_z} (\mathcal{T}f)\big)  \widetilde{h}(z) \widetilde{\mu}(\dd{z}) \\
      =& \int_{\mathbb{R}^k} \left( \int_{\Sigma_{z}}
	 (\mathcal{T}f)(x) \mu_z(\dd{x}) \right) \widetilde{h}(z)
	 \widetilde{\mu}(\dd{z}) \\
      =&  \int_{\mathbb{R}^d} (\mathcal{T}f)(x) \widetilde{h}(\xi(x)) \mu(\dd{x})  \\
      =&  \langle \mathcal{T}f, h\rangle_{\mu}\,,  
    \end{aligned}
  \end{equation*}
  where $f=\widetilde{f}\circ\xi$ and $h=\widetilde{h}\circ\xi$.

\end{proof}

Using the above relation, we discuss in the following remark the relevance of effective dynamics to data-driven numerical algorithms in VAMP when the system's features are employed. 

\begin{remark}[Relevance of effective dynamics in numerical algorithms]
  Suppose that in applications the VAMP-r score is maximized in order to learn
  the $m$ leading eigenfunctions $(f_i)_{1\le i\le m}$ of $\mathcal{T}$ (see Remark~\ref{rmk-vamp} and~\cite{Wu2020,vampnet}). 
Often, these eigenfunctions are represented as functions of system's features, i.e.\ $f_i(x)=\widetilde{f}_i(\mathbf{R}(x))$, where $\widetilde{f}_i: \mathbb{R}^K\rightarrow \mathbb{R}$, integer $K \gg 1$, and 
  $\mathbf{R}=(\mathbf{R}_1,\dots, \mathbf{R}_{K})^\top:
  \mathbb{R}^d\rightarrow \mathbb{R}^{K}$ is a map describing $K$ features (e.g.\ internal variables) of the system $X_n$.  
Consider this specific CV map $\xi=\mathbf{R}$ and define the quantities in Section~\ref{subsec-formal-defs} accordingly.
Using the second claim in Lemma~\ref{lemma-tran-tran-eff}, we obtain for the VAMP-r score 
  \begin{align}
    \sum_{i=1}^m \langle f_i, \mathcal{T} f_i\rangle_\mu^r = \sum_{i=1}^m
    \langle \widetilde{f}_i, \widetilde{\mathcal{T}} \widetilde{f}_i\rangle_{\widetilde{\mu}}^r \,,
  \end{align}
  and the constraints in \eqref{f-constraints} are equivalent to 
  \begin{equation} 
      \mathbb{E}_{\widetilde{\mu}} (\widetilde{f}_i)=0, \quad \langle \widetilde{f}_i, \widetilde{f}_j\rangle_{\widetilde{\mu}} = \delta_{ij}\,,  \quad \forall\, i,j\, \in \{1,\dots,m\}\,.
    \label{f-constraints-eff}
  \end{equation}
  In other words, approximating eigenfunctions of $\mathcal{T}$ by maximizing VAMP-r score among functions of system's features is equivalent to approximating eigenfunctions of $\widetilde{\mathcal{T}}$ associated to the feature map $\mathbf{R}$.
  \label{vamp-r-eff}
\end{remark}

A similar remark can be made on the relevance of effective dynamics to numerical algorithms for computing committor based on solving the variational
problem in \eqref{min-energy-committor}. In the last part of this section, we discuss the reversibility of $Z_n$ and its connection to the
reversibility of $X_n$. Denote by $\widetilde{\mathcal{T}}^*$ the adjoint of
$\widetilde{\mathcal{T}}$ in $\widetilde{\mathcal{H}}$, which satisfies
$\langle \widetilde{\mathcal{T}}\widetilde{f}, \widetilde{h}\,\rangle_{\widetilde{\mu}} = \langle \widetilde{f},
\widetilde{\mathcal{T}}^*\widetilde{h}\,\rangle_{\widetilde{\mu}}$ for any $\widetilde{f}, \widetilde{h}\in \widetilde{\mathcal{H}}$. 
The following result gives the concrete expression of $\widetilde{\mathcal{T}}^*$.
\begin{proposition}
  For $\widetilde{f} \in \widetilde{\mathcal{H}}$, we have the following identity 
\begin{equation}
  (\widetilde{\mathcal{T}}^*\widetilde{f}\,)(z) = \int_{\mathbb{R}^k}
  \widetilde{p}^*(z,w) \widetilde{f}(w)\, \dd{w}\,, \quad z \in \mathbb{R}^k\,,
  \label{trans-operator-eff-adjoint}
\end{equation}
with the transition density 
\begin{equation}
  \begin{aligned}
  \widetilde{p}^*(z,w) =
    \frac{\widetilde{p}(w,z)\widetilde{\pi}(w)}{\widetilde{\pi}(z)} = &
\int_{\Sigma_{z}} \left(\int_{\Sigma_{w}} p^*(x, y)
       [\mathrm{det}(\nabla\xi(y)^\top\nabla\xi(y))]^{-\frac{1}{2}}
       \sigma_{\Sigma_{w}}(\dd{y}) \right)\mu_{z}(\dd{x}) \\
       =& \widetilde{\pi}(w) \int_{\Sigma_{z}} \left(\int_{\Sigma_{w}}
       \frac{p^*(x, y)}{\pi(y)}
       \mu_{w}(\dd{y}) \right)\mu_{z}(\dd{x})\,, 
  \end{aligned}
  \label{p-z-adjoint}
\end{equation}
for $z,w \in \mathbb{R}^k$, where $p^*$ is given in \eqref{p-adjoint}.
  \label{prop-trans-operator-adjoint}
\end{proposition}
\begin{proof}
For any $\widetilde{f}, \widetilde{h}\in \widetilde{\mathcal{H}}$, direct calculation gives
  \begin{equation*}
    \begin{aligned}
    \langle \widetilde{\mathcal{T}}^*\widetilde{f}, \widetilde{h}\rangle_{\widetilde{\mu}} = 
\langle \widetilde{f}, \widetilde{\mathcal{T}}\widetilde{h}\rangle_{\widetilde{\mu}}
      =&\int_{\mathbb{R}^k} \bigg[\int_{\mathbb{R}^k}  \widetilde{p}(z,w)
      \widetilde{h}(w)\,\dd{w}\bigg]\widetilde{f}(z) \widetilde{\mu}(\dd{z})\\
      =&\int_{\mathbb{R}^k} \bigg[\int_{\mathbb{R}^k}
      \frac{\widetilde{p}(z,w)\widetilde{\pi}(z)}{\widetilde{\pi}(w)}
      \widetilde{f}(z)\, \dd{z}\bigg]
      \widetilde{h}(w)\,\widetilde{\mu}(\dd{w}) \,,
    \end{aligned}
  \end{equation*}
   from which we obtain \eqref{trans-operator-eff-adjoint} with $\widetilde{p}^*$ given by the first equality in \eqref{p-z-adjoint}.
   To prove the second equality in \eqref{p-z-adjoint}, let us derive
   \begin{equation*}
     \begin{aligned}
     \widetilde{p}^*(z,w) =& \frac{\widetilde{p}(w,z)\widetilde{\pi}(w)}{\widetilde{\pi}(z)} \\
=& \widetilde{\pi}(w) 
       \int_{\Sigma_{w}}\left(\int_{\Sigma_{z}} \,\frac{p(y, x)}{\pi(x)}\,\mu_{z}(\dd{x}) \right)\mu_{w}(\dd{y}) \\
=& \widetilde{\pi}(w) 
       \int_{\Sigma_{z}}\left(\int_{\Sigma_{w}} \,\frac{p^*(x, y)}{\pi(y)}\,\mu_{w}(\dd{y}) \right)\mu_{z}(\dd{x}) \\
=& \int_{\Sigma_{z}} \left(\int_{\Sigma_{w}}( p^*(x, y)\,
       [\mathrm{det}(\nabla\xi(y)^\top\nabla\xi(y))]^{-\frac{1}{2}}
       \sigma_{\Sigma_{w}}(\dd{y}) \right)\,\mu_{z}(\dd{x})\,, 
   \end{aligned}
   \end{equation*}
   where the second equality follows from the second equality in \eqref{p-z}, the third equality follows from the definition of $p^*$
   in \eqref{p-adjoint} and an exchange of integrals, the fourth equality
   follows from the definition of $\mu_w$ in \eqref{mu-z}.
\end{proof}

Proposition~\ref{prop-trans-operator-adjoint} implies that
$\widetilde{\mathcal{T}}^*$ is the transfer operator associated to the
effective dynamics of the adjoint process of $X_n$ (the process defined by the transition density $p^*$). As a simple application of Proposition~\ref{prop-trans-operator-adjoint}, we have the following result on the reversibility of $Z_n$.
\begin{corollary}
  Suppose that $X_n$ is reversible. Then, the effective process $Z_n$ is also reversible.
  \label{corollary-reversibility-eff}
\end{corollary}
\begin{proof}
  The fact that $X_n$ is reversible implies $p^*=p$ (see~\eqref{reversibility}). In this case, we obtain
  $\widetilde{p}^*=\widetilde{p}$ by comparing \eqref{p-z-adjoint} to \eqref{p-z},
  and \eqref{trans-operator-eff-adjoint} implies that $\widetilde{\mathcal{T}}^*=\widetilde{\mathcal{T}}$.
  Therefore, $\widetilde{\mathcal{T}}$ is self-adjoint and the effective process $Z_n$ is reversible.
\end{proof}

Before concluding this section, let us introduce the Dirichlet energy
associated to the effective dynamics $Z_n$ (assuming reversibility). 
 Similar to the Dirichlet energy \eqref{energy} associated to $X_n$, for a test function $\widetilde{f}:\mathbb{R}^k\rightarrow \mathbb{R}$, we define its Dirichlet energy associated to $Z_n$ as 
\begin{equation}
  \widetilde{\mathcal{E}}(\widetilde{f}) = \frac{1}{2}
  \int_{\mathbb{R}^k}\left(\int_{\mathbb{R}^k} \widetilde{p}(z,w)
  (\widetilde{f}(w)-\widetilde{f}(z))^2 \dd{w}\right) \widetilde{\mu}(\dd{z})\,.
  \label{energy-eff}
\end{equation}
In analogy to \eqref{e-f-g-adjoint}, applying Lemma~\ref{lemma-tran-tran-eff} we have (for the first identity see the proof of Lemma~\ref{lemma-energy-and-msv})
\begin{equation}
     \widetilde{\mathcal{E}}(\widetilde{f}) = \langle (\mathcal{I}-\widetilde{\mathcal{T}})\widetilde{f}, \widetilde{f}\rangle_{\widetilde{\mu}} =
  \langle (\mathcal{I}-\mathcal{T})(\widetilde{f}\circ\xi),
  \widetilde{f}\circ\xi\rangle_\mu = \mathcal{E}(\widetilde{f}\circ\xi)\,.
  \label{energy-and-its-eff}
\end{equation}
Both \eqref{energy-eff} and \eqref{energy-and-its-eff} will be useful in the next section when we compare eigenvalues and transition rates of $X_n$ to those of $Z_n$.

\section{Error estimates for effective dynamics}
\label{sec-error-estimates}

In this section, we study the approximation of the effective dynamics to the original process $X_n$.
We assume that $X_n$ is reversible and therefore its effective dynamics is also reversible by Corollary~\ref{corollary-reversibility-eff}.
In Section~\ref{subsec-error-timescales} and Section~\ref{subsec-error-rates},
we provide error estimates for the effective dynamics in approximating timescales and transition rates of $X_n$, respectively. 
More importantly, we discuss how these estimates are relevant in the context of optimizing CV maps.
Similar error estimates for effective dynamics of diffusion processes have been obtained in~\cite{effective_dyn_2017}. 
In Section~\ref{subsec-transformation}, we compare the timescales and transition rates of effective dynamics associated to two CV maps.

\subsection{Timescales}
\label{subsec-error-timescales}

First, we compare the timescales of the effective dynamics~$Z_n$ to the timescales of $X_n$.
Similar to the eigenvalue problem \eqref{eigen-problem} in
Section~\ref{subsec-spectrum} associated to $X_n$, we introduce the eigenvalue
problem associated to the effective dynamics $Z_n$ in $\mathbb{R}^k$:
\begin{equation}
  \widetilde{\mathcal{T}}\,\widetilde{f} = \widetilde{\lambda}\,\widetilde{f}\,,
  \label{eigen-problem-eff}
\end{equation}
where 
$\widetilde{\mathcal{T}}$ is the transfer operator in
\eqref{trans-operator-eff}, $\widetilde{\lambda} \in \mathbb{R}$, and the function $\widetilde{f}:\mathbb{R}^k\rightarrow \mathbb{R}$ is normalized such that
$\langle \widetilde{f}, \widetilde{f}\rangle_{\widetilde{\mu}}=1$.
In analogy to the discussions in Section~\ref{subsec-spectrum}, we also assume
that the spectrum of $\widetilde{\mathcal{T}}$ consists of eigenvalues $(\widetilde{\lambda}_i)_{i\ge 0}$ of \eqref{eigen-problem-eff},
which are within the range $(0, 1]$ and can be sorted in a way such that 
\begin{equation}
  1=\widetilde{\lambda}_0 > \widetilde{\lambda}_1 \ge \widetilde{\lambda}_2 \ge \dots \ge 0\,,
  \label{eigen-seq-eff}
\end{equation}
with the corresponding orthonormal eigenfunctions $(\widetilde{\varphi}_i)_{i\ge 0}$, where $\widetilde{\varphi}_i:\mathbb{R}^k\rightarrow \mathbb{R}$ for $i\ge 0$ and $\widetilde{\varphi}_0\equiv 1$.
This assumption is true when the density $\widetilde{p}$ and the transfer
operator $\widetilde{\mathcal{T}}$ satisfy the conditions of Proposition~\ref{prop-spectrum} and Proposition~\ref{prop-nonnegative-spectrum}.

We have the following result comparing the eigenvalues of $\widetilde{\mathcal{T}}$ to the eigenvalues of $\mathcal{T}$.
\begin{theorem}
  For each $i=1,2,\dots$, let $\lambda_i$ be the eigenvalue of
  $\mathcal{T}$ in \eqref{eigen-seq} and $\varphi_i$ be the
  corresponding normalized eigenfunction. Let $\widetilde{\lambda}_i$ be the
  eigenvalue of $\widetilde{\mathcal{T}}$ in \eqref{eigen-seq-eff} and
  $\widetilde{\varphi}_i$ be the corresponding normalized eigenfunction. Let
  $\mathcal{E}$ denote the Dirichlet energy~\eqref{energy} associated to $X_n$.
  The following claims hold.
  \begin{enumerate}
    \item
    $\widetilde{\lambda}_i \le \lambda_i$.
    \item
For any $j=1,2,\dots$, we have $\lambda_i - \widetilde{\lambda}_j = \mathcal{E}(\widetilde{\varphi}_j\circ\xi-\varphi_i) - (1-\lambda_i)\|\widetilde{\varphi}_j\circ\xi-\varphi_i\|^2_\mu$. 
  \item
    Assume that there is a function $\widetilde{h}\in \widetilde{\mathcal{H}}$ such that $\varphi_i = \widetilde{h}\circ \xi$. Then, $\lambda_i$ is an eigenvalue
      of $\widetilde{\mathcal{T}}$ and $\widetilde{h}$ is the corresponding
      eigenfunction, i.e. $\widetilde{\lambda}_j=\lambda_i$ and $\widetilde{\varphi}_j=\widetilde{h}$ for some integer $j\ge 1$.
  \end{enumerate}
  \label{thm-eigen-cmp}
\end{theorem}
\begin{proof}
  The proof is similar to the proof of~\cite[Proposition~5]{effective_dyn_2017}.
  Concerning the first claim, notice that $1-\lambda_i$ is the $(i+1)$th smallest eigenvalue of $\mathcal{I}-\mathcal{T}$ for $i\ge 1$. 
  Therefore, the min-max theorem~\cite[Theorem 4.14]{teschl2009mathematical} implies 
  \begin{equation}
    1-\lambda_i = \inf_{H_i} \sup_{\psi\in H_i}\Big\{\langle (\mathcal{I}-\mathcal{T})\psi, \psi\rangle_\mu, \langle\psi,\psi\rangle_\mu=1\Big\}\,,
    \label{min-max-1}
  \end{equation}
  where $H_i$ goes over all $(i+1)$-dimensional linear subspace of
  $\mathcal{H}$. In particular, let $\widetilde{H}_i$ be any $(i+1)$-dimensional linear subspace of
  $\widetilde{\mathcal{H}}$ and define $H_i=\{\widetilde{f}\circ\xi\,|\,\widetilde{f}\in \widetilde{H}_i\}$.  Then, it is easy to show that $H_i$ is an
  $(i+1)$-dimensional linear subspace of $\mathcal{H}$. Therefore, for such $H_i$, from \eqref{min-max-1} we obtain
  \begin{equation}
    1-\lambda_i \le \sup_{\psi\in H_i}\Big\{\langle (\mathcal{I}-\mathcal{T})\psi, \psi\rangle_\mu, \langle\psi,\psi\rangle_\mu=1\Big\}\,.
    \label{proof-eigen-bound-1}
  \end{equation}
   Applying Lemma~\ref{lemma-tran-tran-eff}, we have 
  \begin{equation}
    \sup_{\psi\in H_i}\Big\{\langle (\mathcal{I}-\mathcal{T})\psi, \psi\rangle_\mu, \langle\psi,\psi\rangle_\mu=1\Big\}
      =\sup_{\widetilde{\psi}\in \widetilde{H}_i}\Big\{\langle
      (\mathcal{I}-\widetilde{\mathcal{T}})\widetilde{\psi},
      \widetilde{\psi}\rangle_{\widetilde{\mu}},
      \langle\widetilde{\psi},\widetilde{\psi}\rangle_{\widetilde{\mu}}=1\Big\}\,.
    \label{proof-identity-1}
  \end{equation}
  Combining \eqref{proof-eigen-bound-1} and \eqref{proof-identity-1}, and applying the min-max theorem to $\mathcal{I}-\widetilde{\mathcal{T}}$, we get 
  \begin{equation*}
      1-\lambda_i \le \inf_{\widetilde{H}_i} \sup_{\widetilde{\psi}\in
      \widetilde{H}_i}\Big\{\langle
      (\mathcal{I}-\widetilde{\mathcal{T}})\widetilde{\psi},
      \widetilde{\psi}\rangle_{\widetilde{\mu}},
      \langle\widetilde{\psi},\widetilde{\psi}\rangle_{\widetilde{\mu}}=1\Big\} 
      = 1 - \widetilde{\lambda}_i\,,
  \end{equation*}
  which implies the first claim. 

  Concerning the second claim, let us derive, for any $j\ge 1$, 
  \begin{equation}
    \begin{aligned}
      1 - \widetilde{\lambda}_j =& \langle
      (\mathcal{I}-\widetilde{\mathcal{T}})\widetilde{\varphi}_j,\widetilde{\varphi}_j\rangle_{\widetilde{\mu}}\\
      =& \langle (\mathcal{I}-\mathcal{T})(\widetilde{\varphi}_j\circ\xi),\widetilde{\varphi}_j\circ\xi\rangle_\mu\\
      =& \langle (\mathcal{I}-\mathcal{T})(\widetilde{\varphi}_j\circ\xi-\varphi_i+\varphi_i),\widetilde{\varphi}_j\circ\xi-\varphi_i + \varphi_i\rangle_\mu\\
      =& \langle (\mathcal{I}-\mathcal{T})(\widetilde{\varphi}_j\circ\xi-\varphi_i),\widetilde{\varphi}_j\circ\xi-\varphi_i\rangle_\mu
      + \langle (\mathcal{I}-\mathcal{T})\varphi_i,\varphi_i\rangle_\mu \\
      & + 2 \langle (\mathcal{I}-\mathcal{T})\varphi_i, \widetilde{\varphi}_j\circ\xi-\varphi_i \rangle_\mu\\
      =& \mathcal{E}(\widetilde{\varphi}_j\circ\xi-\varphi_i) + 1 - \lambda_i + 2(1-\lambda_i) \langle\varphi_i, \widetilde{\varphi}_j\circ\xi-\varphi_i \rangle_\mu\,,\\
    \end{aligned}
    \label{proof-identity-2}
  \end{equation}
  where we have used the fact that $\widetilde{\varphi}_i$ is normalized to get the first equality, Lemma~\ref{lemma-tran-tran-eff} to obtain the second
  equality, and the fact that $\varphi_i$ is the eigenfunction of $\mathcal{I}-\mathcal{T}$ corresponding to the eigenvalue $1-\lambda_i$ to arrive at the last equality.
  For the last term on the last line of \eqref{proof-identity-2}, we notice that
  \begin{equation}
        2\langle\varphi_i, \widetilde{\varphi}_j\circ\xi-\varphi_i \rangle_\mu = 2\langle\varphi_i, \widetilde{\varphi}_j\circ\xi\rangle_\mu - 2 = -\langle\widetilde{\varphi}_j\circ\xi-\varphi_i, \widetilde{\varphi}_j\circ\xi-\varphi_i \rangle_\mu=-\|\widetilde{\varphi}_j\circ\xi-\varphi_i\|_\mu^2\,,
    \label{proof-identity-3}
  \end{equation}
  which follows from the fact that both eigenfunctions $\varphi_i$ and
  $\widetilde{\varphi}_j$ are normalized, and the equality
  $\|\widetilde{\varphi}_j\circ\xi\|_\mu =
  \|\widetilde{\varphi}_j\|_{\widetilde{\mu}}=1$ as implied by
  Lemma~\ref{lemma-tran-tran-eff}. The identity in the second claim is obtained by combining and reorganizing \eqref{proof-identity-2}--\eqref{proof-identity-3}.

  Concerning the third claim, using the identity in the first claim of Lemma~\ref{lemma-tran-tran-eff} we obtain 
  $(\widetilde{\mathcal{T}}\widetilde{h})(z) = \mathbb{E}_{\mu_z} (\mathcal{T}(\widetilde{h}\circ \xi)) = \mathbb{E}_{\mu_z} (\mathcal{T}\varphi_i)
=\lambda_i \mathbb{E}_{\mu_z} (\varphi_i) =\lambda_i \widetilde{h}(z)$, from which we can conclude.
\end{proof}

Let us discuss the implications of Theorem~\ref{thm-eigen-cmp}. The first
claim of Theorem~\ref{thm-eigen-cmp} implies that the eigenvalues of $Z_n$ are
smaller than (or equal to) the corresponding eigenvalues of $X_n$. In other
words, the convergence of the effective dynamics to equilibrium is faster than
(or equal to) the convergence of the original process (see
\eqref{eqn-evolution-f-by-eigenfunc}).  The second claim of
Theorem~\ref{thm-eigen-cmp} implies that the true eigenvalue $\lambda_i$ (or
equivalently, timescale) of $X_n$ can be well approximated by an eigenvalue of
$Z_n$, if the corresponding eigenfunction $\varphi_i$ is approximately a
function of the CV map $\xi$. In particular, if a component of $\xi$ is capable of representing the eigenfunction $\varphi_i$, e.g.\
$\varphi_i=\xi_1$, then the third claim of Theorem~\ref{thm-eigen-cmp} implies
that the corresponding eigenvalue $\lambda_i$ is preserved by the effective
dynamics. Recall that, in data-driven numerical approaches for timescales estimations, employing system's features corresponds to estimating
timescales of effective dynamics  (see Remark~\ref{vamp-r-eff}). Therefore,
Theorem~\ref{thm-eigen-cmp} suggests that, in order to obtain accurate timescale estimations,
features should be selected such that they can parametrize the corresponding eigenfunctions.

We conclude this section with a theorem on a numerical approach for learning the CV map $\xi$ by optimizing timescale estimations.

\begin{theorem}[Optimizing CV maps by timescale estimations]\label{thm:opt_timescales}
  The variational formulation \eqref{variational-transfer-all-first-m}--\eqref{f-constraints} in Theorem~\ref{thm-variational-form-transfer}
can be used to learn the CV map $\xi$, by optimizing
the objective in \eqref{variational-transfer-all-first-m} with functions $f_1,
  \dots, f_m$, which are parametrized by a common CV map $\xi$,
  i.e.\ $f_i=\widetilde{f}_i\circ\xi$ for some function $\widetilde{f}_i: \mathbb{R}^k\rightarrow \mathbb{R}$, where $1\le i \le m$.
  In this case, we have 
  \begin{equation}
    \min_{\substack{f_i=\widetilde{f}_i\circ\xi, \xi\in \Xi\\ \widetilde{f}_i\in
    \widetilde{\mathcal{H}}, 1\le i\le m}} \sum_{i=1}^m \omega_i \mathcal{E}(f_i) = \min_{\xi\in \Xi}
    \min_{\widetilde{f}_1,\dots, \widetilde{f}_m\in \widetilde{\mathcal{H}}} \sum_{i=1}^m \omega_i \mathcal{E}(\widetilde{f}_i\circ\xi) =
    \min_{\xi\in \Xi} \sum_{i=1}^m \omega_i (1 - \widetilde{\lambda}_i),
    \label{opt_timescales-eqn1}
  \end{equation}
  where the functions $\widetilde{f}_i$ satisfy the constraints \eqref{f-constraints-eff}.
\end{theorem}
\begin{proof}
The constraints \eqref{f-constraints-eff} on $\widetilde{f}_i$ follow from the constraints in \eqref{f-constraints} and Lemma~\ref{lemma-tran-tran-eff}.
To prove the second equality in \eqref{opt_timescales-eqn1}, we derive 
  \begin{equation*}
    \min_{\xi\in \Xi} \min_{\widetilde{f}_1,\dots, \widetilde{f}_m\in \widetilde{\mathcal{H}}} \sum_{i=1}^m \omega_i \mathcal{E}(\widetilde{f}_i\circ\xi) =
    \min_{\xi\in \Xi}\min_{\widetilde{f}_1,\dots, \widetilde{f}_m\in \widetilde{\mathcal{H}}} \sum_{i=1}^m \omega_i \widetilde{\mathcal{E}}(\widetilde{f}_i) = \min_{\xi\in \Xi} \sum_{i=1}^m \omega_i (1 - \widetilde{\lambda}_i),
  \end{equation*}
where the first equality follows from \eqref{energy-and-its-eff} and the second equality follows from the variational formulation in
  Theorem~\ref{thm-variational-form-transfer} when applied to the eigenfunctions of the effective dynamics.
\end{proof}
  
  \begin{remark}[Interpretation and algorithmic realization]
    The rightmost expression in \eqref{opt_timescales-eqn1} shows that optimizing functions $f_1,\dots, f_m$ of the form
  $f_i=\widetilde{f}_i\circ \xi$ for $1\le i \le m$ by minimizing 
  the objective in \eqref{variational-transfer-all-first-m} (under the
  constraints \eqref{f-constraints}) is equivalent to optimizing the CV map $\xi$ such that the quantity $\sum_{i=1}^m \omega_i
  (1-\lambda_i)$ associated to $X_n$ is best approximated by the quantity $\sum_{i=1}^m
  \omega_i (1-\widetilde{\lambda}_i)$ associated to the effective dynamics
  built using the map $\xi$. We refer to \cite[Section 2.7]{lelievre2023analyzing} for a numerical algorithm that learns an encoder (i.e.\ the map $\xi$) using a loss function that combines reconstruction error and the objective in \eqref{variational-transfer-all-first-m}.
  \label{rmk-opt-xi-timescales}
\end{remark}

\subsection{Transition rates}
\label{subsec-error-rates}

Next, we compare the transition rates of $Z_n$ to the corresponding transition rates of $X_n$.
To begin with, recall that in TPT the transition rate $k_{AB}$ of $X_n$ given two disjoint closed subsets $A,B\subset \mathbb{R}^d$ is defined in \eqref{k-ab-def} in Section~\ref{subsec-tpt}. To make a connection to the effective dynamics, let us furthermore assume that these two subsets are defined using the CV map $\xi$, i.e.\ there are two disjoint subsets $\widetilde{A}, \widetilde{B}
\subset \mathbb{R}^k$, such that~(see Figure~\ref{fig-tpt} for illustration)
\begin{equation}
A=\xi^{-1}(\widetilde{A}), \quad B=\xi^{-1}(\widetilde{B}), \quad
  (\widetilde{A}\cup \widetilde{B})^c\neq \emptyset\,. 
  \label{set-ab-and-eff}
\end{equation}
Then, in analogy to \eqref{k-ab-def}, we define the transition rate of $Z_n$ from $\widetilde{A}$ to $\widetilde{B}$ as 
\begin{equation}
  \widetilde{k}_{\widetilde{A}\widetilde{B}} = \lim_{N\rightarrow +\infty} \frac{\widetilde{M}_N^R}{N}\,,
  \label{k-ab-def-eff}
\end{equation}
by counting the number of reactive segments $\widetilde{M}_N^R$ in the
trajectory of $Z_n$ within the first $N$ steps and taking the limit $N\rightarrow +\infty$.

\begin{figure}[t!]
  \centering
  \includegraphics[width=0.7\textwidth]{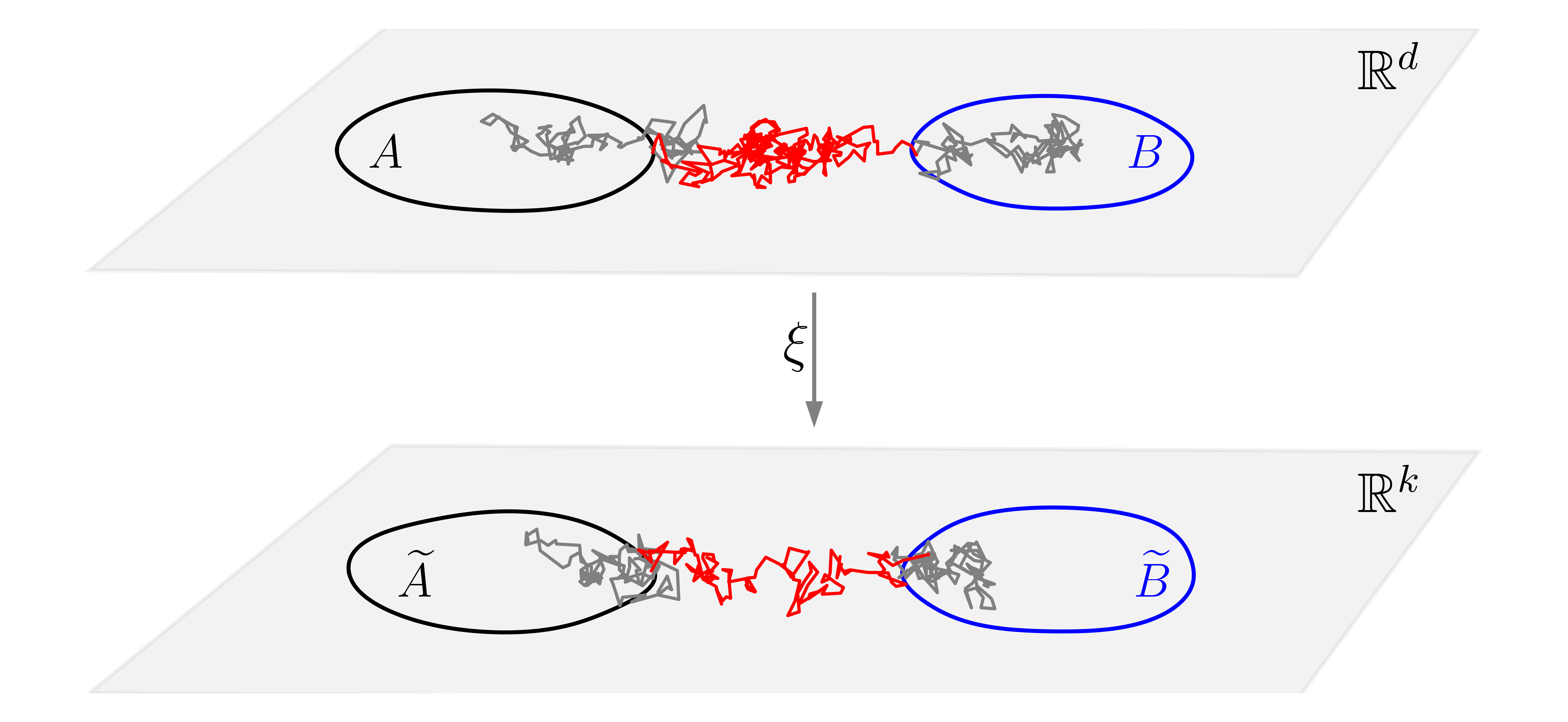}
  \caption{Top: illustration of transitions of the process $X_n$ from set $A$ to set $B$ in $\mathbb{R}^d$. Bottom: illustration of
  transitions of the system $Z_n$ from the corresponding set $\widetilde{A}$
  to set $\widetilde{B}$ in $\mathbb{R}^k$ (see \eqref{set-ab-and-eff} for the relations between the sets). The reactive segments are highlighted in red. 
  \label{fig-tpt}}
\end{figure}

In analogy to the third claim of Proposition~\ref{prop-kab}, for the effective dynamics we have  
\begin{equation}
  \widetilde{k}_{\widetilde{A}\widetilde{B}} =
  \widetilde{\mathcal{E}}(\widetilde{q}) = \frac{1}{2}\int_{\mathbb{R}^k}
  \left(\int_{\mathbb{R}^k} \widetilde{p}(z,w) (\widetilde{q}(w) -
  \widetilde{q}(z))^2 \dd{w}\right) \widetilde{\mu}(\dd{z})\,,
  \label{kab-and-energy-eff}
\end{equation}
where $\widetilde{\mathcal{E}}$ is the Dirichlet energy in \eqref{energy-eff} and
$\widetilde{q}$ is the (forward) committor associated to the transitions of $Z_n$ from $\widetilde{A}$ to $\widetilde{B}$, satisfying the equation
\begin{subequations}
  \begin{align}
    & (\widetilde{\mathcal{T}}\widetilde{q})(z) = \widetilde{q}(z), \quad \forall~ z \in (\widetilde{A}\cup \widetilde{B})^c\,, \label{eqn-committor-eff-1} \\
    & \widetilde{q}\,|_{\widetilde{A}} \equiv 0, \quad \widetilde{q}\,|_{\widetilde{B}} \equiv 1\,.\label{eqn-committor-eff-2}
  \end{align}
\end{subequations}

We have the following result which compares the transition rate $k_{AB}$ of
the original process $X_n$ to the transition rate $\widetilde{k}_{\widetilde{A}\widetilde{B}}$ of the effective dynamics $Z_n$.
\begin{theorem}\label{thm:rates_CV}
  Assume that the sets $A, B \subset \mathbb{R}^d$ are related to the sets
  $\widetilde{A}, \widetilde{B} \subset \mathbb{R}^k$ through the CV map $\xi$ by \eqref{set-ab-and-eff}.
  Let $q$ and $k_{AB}$ denote the (forward) committor and the transition rate of $X_n$ from $A$ to $B$, respectively. Also, let $\widetilde{q}$ and
  $\widetilde{k}_{\widetilde{A}\widetilde{B}}$ be the committor and the
  transition rate of $Z_n$ from $\widetilde{A}$ to $\widetilde{B}$, respectively.
  Then, we have 
  \begin{equation}
    \widetilde{k}_{\widetilde{A}\widetilde{B}} = k_{AB} + \mathcal{E}(q-\widetilde{q}\circ\xi)\,,
    \label{kab-and-eff}
  \end{equation}
  where $\mathcal{E}$ is the Dirichlet energy in \eqref{energy} associated to $X_n$.
  In particular, if there is a function $\widetilde{f}:\mathbb{R}^k\rightarrow \mathbb{R}$ such that $q=\widetilde{f}\circ \xi$,
  then $\widetilde{f}=\widetilde{q}$ and $\widetilde{k}_{\widetilde{A}\widetilde{B}}=k_{AB}$.
  \label{thm-compare-kab-and-eff}
\end{theorem}
\begin{proof}
We prove the claims by applying Proposition~\ref{prop-committor}. Consider the
  function $f:\mathbb{R}^d\rightarrow \mathbb{R}$ defined by
  $f=\widetilde{q}\circ\xi$. The conditions in \eqref{eqn-committor-eff-2} and
  the relation \eqref{set-ab-and-eff} imply that $f\in \mathcal{F}_{AB}$ (see \eqref{f-ab}).
  The equation \eqref{kab-and-eff} follows from the identity \eqref{energy-f-q} in Proposition~\ref{prop-committor} and the fact that
  $\mathcal{E}(f)=\mathcal{E}(\widetilde{q}\circ\xi) = \widetilde{\mathcal{E}}(\widetilde{q})
  = \widetilde{k}_{\widetilde{A}\widetilde{B}}$, which in turn is implied by \eqref{energy-and-its-eff} and \eqref{kab-and-energy-eff}.

  Concerning the second claim, we notice that the same conclusion of
  Proposition~\ref{prop-committor} holds for $Z_n$, i.e.\ $\widetilde{q}$
  minimizes the Dirichlet energy $\widetilde{\mathcal{E}}(\widetilde{f})$
  among all $\widetilde{f}:\mathbb{R}^k\rightarrow \mathbb{R}$ such that
  $\widetilde{f}|_{\widetilde{A}} \equiv 0$ and
  $\widetilde{f}|_{\widetilde{B}} \equiv 1$. Using this fact and \eqref{energy-and-its-eff}, we can derive
 \begin{equation*}
   \widetilde{k}_{\widetilde{A}\widetilde{B}} =
   \widetilde{\mathcal{E}}(\widetilde{q}) \le
   \widetilde{\mathcal{E}}(\widetilde{f}) = \mathcal{E}(\widetilde{f}\circ\xi)
   = \mathcal{E}(q) = k_{AB}\,,
 \end{equation*}
  which, together with \eqref{kab-and-eff} and the definition of $\mathcal{E}$ in \eqref{energy}, implies that
  $\widetilde{k}_{\widetilde{A}\widetilde{B}} =k_{AB}$ and $\widetilde{f}\circ
  \xi = q=\widetilde{q}\circ \xi$. The proof is therefore completed (since $\xi$ is a onto map).
\end{proof}

Theorem~\ref{thm-compare-kab-and-eff} states that the transition rate $k_{AB}$ of $X_n$ can be well approximated by the
corresponding transition rate $\widetilde{k}_{\widetilde{A}\widetilde{B}}$ of
$Z_n$, if the committor $q$ is approximately a function of the CV map $\xi$.
In particular, the rate $k_{AB}$ is preserved by the effective dynamics if the CV map $\xi$ is capable of parametrizing the committor $q$.

\subsection{Comparison under transformations}
\label{subsec-transformation}

Finally, we compare timescales and transition rates for effective dynamics associated to two CV maps. Besides the CV map $\xi:\mathbb{R}^d\rightarrow\mathbb{R}^k$ and the corresponding effective dynamics $Z_n$, let us consider another CV map $\xi':\mathbb{R}^d\rightarrow \mathbb{R}^{k'}$, which is related to $\xi$ by
$\xi'=f\circ \xi$, where $f: \mathbb{R}^k\rightarrow \mathbb{R}^{k'}$ and $1 \le k' \le k < d$. 
This may correspond to the case discussed in Remark~\ref{vamp-r-eff} where a CV map $\xi'$ is built using system's features (described by $\xi$). 

Denote $Z_n'$ the effective dynamics of $X_n$ associated to $\xi'$, which is defined by the transition density (see~\eqref{p-z} in Definition~\ref{def-effdyn})
  \begin{equation}
    \widetilde{p}\,'(z', w') 
= \int_{\Sigma_{z'}}\Big[ \int_{\Sigma_{w'}}
    p(x, y) [\mathrm{det}(\nabla\xi'(y)^\top\nabla\xi'(y))]^{-\frac{1}{2}}
    \sigma_{\Sigma_{w'}}(\dd{y}) \Big]\,\mu_{z'}(\dd{x})\,, 
    \label{transition-density-p-prime}
  \end{equation}
where $z',w'\in \mathbb{R}^{k'}$.
Since we assume $X_n$ to be reversible,
Corollary~\ref{corollary-reversibility-eff} implies that $Z_n'$ is reversible 
with respect to its invariant probability measure
$\widetilde{\mu}'(\dd{z'})=\widetilde{\pi}'(z')\dd{z'}$, where $\widetilde{\pi}'(z') = 
\int_{\Sigma_{z'}} \pi(x)
[\mathrm{det}(\nabla\xi'(x)^\top\nabla\xi'(x))]^{-\frac{1}{2}}
\,\sigma_{\Sigma_{z'}}(\dd{x})$ (see~\eqref{q-eff}).
Let $\widetilde{\mathcal{T}}'$ be the corresponding transfer operator of $Z_n'$,
with (nonnegative) eigenvalues $1=\widetilde{\lambda}'_0 > \widetilde{\lambda}'_1 \ge \widetilde{\lambda}'_2 \ge \dots$, 
and eigenfunctions $(\widetilde{\varphi}'_i)_{i\ge 0}$, which are normalized with respect to the norm $\|\cdot\|_{\widetilde{\mu}'}$.  
Given the subsets $\widetilde{A},\widetilde{B}\subset \mathbb{R}^k$ that were considered in the previous section (see~\eqref{set-ab-and-eff}), we assume that 
there are two subsets $\widetilde{A}', \widetilde{B}' \subset \mathbb{R}^{k'}$, such that $\widetilde{A}=f^{-1}(\widetilde{A}')$,
$\widetilde{B}=f^{-1}(\widetilde{B}')$ and $(\widetilde{A}' \cup \widetilde{B}')^c\neq\emptyset$.
Let $\widetilde{k}'_{\widetilde{A}'\widetilde{B}'}$ denote the transition rate
of $Z_n'$ from the set $\widetilde{A}'$ to the set $\widetilde{B}'$.

The following result elucidates the relation between the effective dynamics $Z_n$ and $Z_n'$ (see Figure~\ref{fig-effdyn}), and also the relations between their timescales and transition rates. 
\begin{proposition}
  Assume that $1\le k'< k$. The following claims are true.
  \begin{enumerate}
    \item
      $Z_n'$ is the effective dynamics of $Z_n$ associated to the CV map $f$. 
\item
  $\widetilde{\lambda}_i' \le \widetilde{\lambda}_i$, for $i \ge 1$.
\item
   $\widetilde{k}'_{\widetilde{A}'\widetilde{B}'} \ge \widetilde{k}_{\widetilde{A}\widetilde{B}}$.
  \end{enumerate}
\end{proposition}
\label{thm-compare-two-cvs}
\begin{proof}
  Concerning the first claim, for two test functions $\widetilde{h}'_1, \widetilde{h}'_2: \mathbb{R}^{k'}\rightarrow \mathbb{R}$,
  using \eqref{transition-density-p-prime} and the formula \eqref{co-area} we can derive
  \begin{equation}
  \begin{aligned}
    & \int_{\mathbb{R}^{k'}}\left(\int_{\mathbb{R}^{k'}} \widetilde{p}\,'(z',
    w')  \widetilde{h}'_1 (w') \dd{w'}\right) \widetilde{h}'_2(z') \widetilde{\mu}'(\dd{z'}) \\
      = & \int_{\mathbb{R}^d}\left(\int_{\mathbb{R}^d} p(x,y)
      \widetilde{h}_1'(\xi'(y)) \dd{y}\right) \widetilde{h}_2'(\xi'(x)) \mu(x)\\
      = & \int_{\mathbb{R}^d}\left(\int_{\mathbb{R}^d} p(x,y)
      \widetilde{h}_1'(f(\xi(y))) \dd{y}\right) \widetilde{h}_2'(f(\xi(x))) \mu(x)\\
    =& \int_{\mathbb{R}^{k}}\left(\int_{\mathbb{R}^{k}} 
    \widetilde{p}(z, w) \widetilde{h}'_1(f(w)) \dd{w}\right) \widetilde{h}'_2 (f(z)) \widetilde{\mu}(\dd{z})\,.
    \end{aligned}
    \label{p-xi-prime-xi}
  \end{equation}
  In particular, taking $\widetilde{h}'_1 \equiv 1$, the above identities simplify to 
  \begin{equation}
    \int_{\mathbb{R}^{k'}} \widetilde{h}'_2 (z') \widetilde{\mu}'(\dd{z'})
    =\int_{\mathbb{R}^{k}}  \widetilde{h}'_2 (f(z)) \widetilde{\mu}(\dd{z})\,,
    \label{q-xi-prime-xi}
  \end{equation}
  which implies that $\widetilde{\mu}'$ is the pushforward measure of $\widetilde{\mu}$ under $f$.
  Notice that $\widetilde{p}$ and $\widetilde{\mu}$ are the transition density
  and the invariant measure of $Z_n$ respectively. Using the definition of
  effective dynamics (see \eqref{q-eff} and Definition~\ref{def-effdyn}) and
  rewriting the last expression in \eqref{p-xi-prime-xi} using co-area
  formula~\cite[Lemma~3.2]{lelievre2010free}, we conclude that $Z_n'$ is the effective dynamics of $Z_n$ associated to the CV map $f$.
  
  The inequalities in the second and the third claims follow by applying the
  first claim, Theorem~\ref{thm-eigen-cmp} and Theorem~\ref{thm-compare-kab-and-eff} to the dynamics $Z_n$ with the CV map $f$.
\end{proof}

\begin{figure}[t!]
  \centering
  \includegraphics[width=0.3\textwidth]{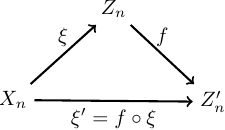}
  \caption{Commutative diagram of effective dynamics. The effective dynamics
  $Z_n'$ of $X_n$ associated to the CV map $\xi'$ is the effective dynamics of
  (the effective dynamics) $Z_n$ associated to the CV map $f$. \label{fig-effdyn}}
\end{figure}

\section{Conclusion and Discussions}
\label{sec-discuss}

One of the ultimate goals in complex system science is to find a faithful
low-dimensional representation of the system's dynamics. The effective
dynamics given by (a set of) collective variables as defined here offers such
a low-dimensional representation.   In this paper, we have studied this
effective dynamics for discrete-in-time Markov processes and presented a
detailed study on its properties, including invariances and approximation
quality with respect to main timescales, transition rates, and KL-distance to the full dynamics. The results characterize which CVs are optimal regarding the respective minimization of these errors.

There are manifold connections between the work presented herein and similar
approaches proposed in the literature. For example, the effective dynamics of
diffusion processes described by SDEs have been studied in several works and
different types of error estimates were obtained, e.g.,
in~\cite{LEGOLL2017-pathwise,effective_dyn_2017,LegollLelievreSharma18,DLPSS18,LelievreZhang18}.
The error estimates in Section~\ref{sec-error-estimates} of this work are closely related to the estimates in~\cite{effective_dyn_2017}. Beyond this, the quality of effective dynamics in approximating the original process can also be studied by considering the relative entropy or pathwise error between the effective dynamics and the original process, as is done, e.g., in~\cite{LEGOLL2017-pathwise,LegollLelievreSharma18,DLPSS18,LelievreZhang18} for diffusion processes.
However, to our best knowledge, there are no similar results on the optimal
characterization for effective dynamics of diffusion processes in terms of relative entropy.

While the main results of this work are theoretical, we discussed the
relevance of effective dynamics in data-driven numerical approaches, such as
in VAMPnets~\cite{vampnet,state-free-vampnets}, deep learning of transition manifolds \cite{optimal-rc-bittracher} and existing algorithms for learning CV maps~\cite{hao-rc-flow,lelievre2023analyzing}. 
The approach taken herein, however, may open the horizon for algorithmic approaches for learning ``good" CVs and the associated effective dynamics \emph{simultaneously} where the already learned effective dynamics might also be utilized for exploring state space,  a key step that is still the main bottleneck of available algorithmic approaches. In future work, this approach will be further studied and applied to concrete systems in MD applications.

\section*{Acknowledgments}
\hspace{0.1cm} CS is funded by the Deutsche Forschungsgemeinschaft (DFG, German Research Foundation) under Germany's Excellence Strategy MATH+: The Berlin Mathematics Research Centre (EXC-2046/1, project No.\,390685689) and CRC 1114 ``Scaling Cascades in Complex Systems'' (project No.\,235221301). WZ is funded by DFG's
Eigene Stelle (project No.\,524086759).
WZ thanks Tony Leli\`evre and Gabriel Stolz for inspiring discussions on topics related to collective variables. 

\appendix 

\section{Transfer operator for Langevin dynamics}
\label{app-sec-langevin}
In this section, we derive the expression of transfer operator for the
evolution of the positional variables of Langevin dynamics (see Example~\ref{example-concrete} in Section~\ref{subsec-transfer-definition}). 

The transition density $p^{\mathrm{Lan}}$ of the Langevin dynamics~\eqref{sde-langevin} satisfies the Fokker-Planck equation 
\begin{align}
  \begin{split}
    \frac{\dd{}}{\dd{t}} p^{\mathrm{Lan}}(t,y,v'|x,v) =& -
    \mathrm{div}_{y}(v'\,p^{\mathrm{Lan}}) + \mathrm{div}_{v'}(\nabla V(y)
    p^{\mathrm{Lan}}) \\
    & + \gamma \mathrm{div}_{v'}(v'\,p^{\mathrm{Lan}}) + \frac{\gamma}{\beta} \Delta_{v'} p^{\mathrm{Lan}}, \quad t > 0\\
    p^{\mathrm{Lan}}(0,y,v'|x,v) =& \delta(y-x)\delta(v'-v)\,,
  \end{split}
  \label{fpe-langevin}
\end{align}
where $t\in [0,+\infty)$, $x,y,v,v'\in \mathbb{R}^d$, and $\mathrm{div}_y$,
$\mathrm{div}_{v'}$, $\Delta_{v'}$ denote the divergence with respect
to $y$, the divergence with respect to $v'$, the Laplacian with respect to $v'$, respectively. Under certain
conditions on $V$, the process is ergodic and the transition density $p^{\mathrm{Lan}}$ converges to the invariant density
$p^{\mathrm{Lan}}_{\infty}(x,v) = Z^{-1}
(\frac{\beta}{2\pi})^{\frac{d}{2}}\mathrm{e}^{-\beta
\big(V(x)+\frac{|v|^2}{2}\big)}$ as $t\rightarrow +\infty$, where $Z=\int_{\mathbb{R}^d} \mathrm{e}^{-\beta V(x)}\,\dd{x}$ is a
normalizing constant~\cite{villani2009hypocoercivity,eberle2019,time-inhomogeneous-2020,Iacobucci2019}.

Suppose that we want to construct a Markov chain for the data $(X_n)_{0 \le n
\le N}$, where $X_{n}=x(n\tau)$ is the position of \eqref{sde-langevin} at time $t=n\tau$ and $\tau>0$ is the lag-time.
Let $p(x,y)$ be the transition density of the Markov chain estimated from the data when the length $N\rightarrow +\infty$.
Given $x, y\in \mathbb{R}^d$, let us denote by $D_1, D_2$ small neighborhoods of $x$ and $y$, respectively.
 We can derive 
\begin{align*}
  & \int_{D_2} p(x,\bar{y})\,\dd{\bar{y}} \\
  = & \lim_{|D_1|\rightarrow
  0}\frac{1}{|D_1|}\int_{D_2}\int_{D_1} p(\bar{x},\bar{y}) \,\dd{\bar{x}}\, \dd{\bar{y}}\\
  = & \lim_{|D_1|\rightarrow 0} \lim_{N\rightarrow +\infty}\frac{\frac{1}{N}\sum_{n=0}^{N-1}\mathbbm{1}_{D_1}(X_n)\mathbbm{1}_{D_2}(X_{n+1})}{\frac{1}{N}\sum_{n=0}^{N-1}\mathbbm{1}_{D_1}(X_n)}\\
  =&\big(\frac{\beta}{2\pi}\big)^{\frac{d}{2}}\lim_{|D_1|\rightarrow 0} \frac{\int_{D_2}\int_{D_1} \int_{\mathbb{R}^d}\int_{\mathbb{R}^d}
  p^{\mathrm{Lan}}(\tau, \bar{y}, v'|\bar{x},v)\,\mathrm{e}^{-\beta
  (V(\bar{x})+\frac{|v|^2}{2})}\,\dd{v}\,\dd{v'}\,\dd{\bar{x}}\,\dd{\bar{y}}}{\int_{D_1}
  \mathrm{e}^{-\beta V(\bar{x})}\,\dd{\bar{x}}}\\
  =& \big(\frac{\beta}{2\pi}\big)^{\frac{d}{2}}\int_{D_2}
  \Big[\int_{\mathbb{R}^d}\int_{\mathbb{R}^d} p^{\mathrm{Lan}}(\tau, \bar{y},
  v'|x,v)\,\mathrm{e}^{-\frac{\beta|v|^2}{2}}\,\dd{v}\,\dd{v'}\Big]\,\dd{\bar{y}}\,,
\end{align*}
where the first and the last equalities follow from the continuity of densities $p$ and $p^{\mathrm{Lan}}$, respectively, 
the second equality follows from the definition of $p$ (i.e.\ the way in which $p$ is estimated from data), the third equality follows from 
Birkhoff's ergodic theorem (see~\cite[Theorem~4.4 in Chapter 1]{Krengel+1985}) and the ergodicity of the process. Since the above derivation is true for any neighborhood $D_2$ of $y$, we conclude that 
\begin{align}
  p(x,y) = \big(\frac{\beta}{2\pi}\big)^{\frac{d}{2}}
  \int_{\mathbb{R}^d}\int_{\mathbb{R}^d} p^{\mathrm{Lan}}(\tau, y, v'|x,v)\,\mathrm{e}^{-\frac{\beta|v|^2}{2}}\,\dd{v}\,\dd{v'}\,.
  \label{trans-langevin-p}
\end{align}

Next, we show that $p(x,y)$ satisfies the detailed balance condition with respect to the density $\frac{1}{Z} \mathrm{e}^{-\beta V}$:
\begin{equation}
  p(x,y)\,\mathrm{e}^{-\beta V(x)} = p(y,x)\,\mathrm{e}^{-\beta V(y)}\,, \quad \forall~x,y \in \mathbb{R}^d\,.
  \label{detailed-balance-trans-langevin}
\end{equation}
For this purpose, notice that the Langevin dynamics is reversible up to the
momentum (velocity) flip. In fact, we have, for $t > 0$,
\begin{equation}
  p^{\mathrm{Lan}}(t,y,v'|x,v)\,p^{\mathrm{Lan}}_{\infty}(x,v) =
  p^{\mathrm{Lan}}(t,x,-v|y,-v')\, p^{\mathrm{Lan}}_{\infty}(y,-v') \,, \quad \forall x,y,v,v'\in \mathbb{R}^d,   
  \label{detailed-balance-momentum-flip}
\end{equation}
which can be proven using \eqref{fpe-langevin} and the Kolmogorov backward
equation satisfied by $p^{\mathrm{Lan}}$ when considered as a function of $(t,x,v)$.
Applying \eqref{trans-langevin-p} and \eqref{detailed-balance-momentum-flip}, we can derive 
\begin{align*}
   p(x,y)\,\frac{1}{Z}\mathrm{e}^{-\beta V(x)} =&\frac{1}{Z}
  \big(\frac{\beta}{2\pi}\big)^{\frac{d}{2}}
  \int_{\mathbb{R}^d}\int_{\mathbb{R}^d} p^{\mathrm{Lan}}(\tau, y,
  v'|x,v)\,\mathrm{e}^{-\beta \big(V(x)+\frac{|v|^2}{2}\big)}\,\dd{v}\,\dd{v'}
  \\
  =& \int_{\mathbb{R}^d}\int_{\mathbb{R}^d} p^{\mathrm{Lan}}(\tau, y, v'|x,v)\, p^{\mathrm{Lan}}_{\infty}(x,v) \,\dd{v}\,\dd{v'} \\
  =& \int_{\mathbb{R}^d}\int_{\mathbb{R}^d} p^{\mathrm{Lan}}(\tau, x, -v|y,-v')\, p^{\mathrm{Lan}}_{\infty}(y,-v') \,\dd{v}\,\dd{v'} \\
  =& \frac{1}{Z}\big(\frac{\beta}{2\pi}\big)^{\frac{d}{2}}
  \int_{\mathbb{R}^d}\int_{\mathbb{R}^d} p^{\mathrm{Lan}}(\tau, x,
  -v|y,-v')\,\mathrm{e}^{-\beta
  \big(V(y)+\frac{|v'|^2}{2}\big)}\,\dd{v}\,\dd{v'}\\
  =&p(y,x)\,\frac{1}{Z} \mathrm{e}^{-\beta V(y)} \,,
\end{align*}
which gives \eqref{detailed-balance-trans-langevin}.

\section{Proofs in Section~\ref{sec-tran-operator}}
\label{app-sec-proofs}

In this section, we present the proofs of Proposition~\ref{prop-spectrum} and Theorem~\ref{thm-variational-form-transfer} in Section~\ref{subsec-spectrum}, as well as the proof of Proposition~\ref{prop-kab}
in Section~\ref{subsec-tpt}.

\begin{proof}[Proof of Proposition~\ref{prop-spectrum}]
  To prove the first claim, let us apply~\cite[Theorem VI.23]{reed1981functional} and show that~$\mathcal{T}$ is a Hilbert-Schmidt operator if \eqref{condition-hilbert-schmidt} holds.
  Obviously, by the definition \eqref{trans-operator} we have 
  \begin{equation*}
  \mathcal{T}f(x) = \int_{\mathbb{R}^d} p(x,y) f(y)\, \dd{y} = \int_{\mathbb{R}^d} K(x,y)  f(y)\,\mu(\dd{y})\,,
  \end{equation*}
    for $f\in \mathcal{H}$, where $K(x,y)=\frac{p(x,y)}{\pi(y)}$.
  And, using \eqref{condition-hilbert-schmidt} and the detailed balance condition \eqref{reversibility}, we have 
  \begin{equation}
    \begin{aligned}
      &\int_{\mathbb{R}^d} \int_{\mathbb{R}^d} |K(x,y)|^2 \,\mu(\dd{x})\,\mu(\dd{y})\\
      =& \int_{\mathbb{R}^d} \int_{\mathbb{R}^d}
      \left|\frac{p(x,y)}{\pi(y)}\right|^2 \pi(x)\pi(y) \,\dd{x}\,\dd{y}
      = \int_{\mathbb{R}^d} \int_{\mathbb{R}^d} p(x,y)p(y,x) \,\dd{x}\,\dd{y} < +\infty\,.
    \end{aligned}
  \end{equation}
  Therefore, the conditions of \cite[Theorem VI.23]{reed1981functional} hold and it follows that $\mathcal{T}$ is a Hilbert-Schmidt operator.
  This in turn implies that $\mathcal{T}$ is compact and any nonzero element in its spectrum $\sigma(\mathcal{T})$ is an eigenvalue (see~\cite[Theorem VI.22]{reed1981functional} and~\cite[Theorem VI.15]{reed1981functional}).

Next, we prove the second claim. Assume without loss of generality that the eigenfunction $\varphi$ corresponding to the eigenvalue $\lambda$ is normalized such that $\langle \varphi, \varphi\rangle_\mu=1$. Multiplying both sides of \eqref{eigen-problem} by $\varphi$ and integrating with respect to $\mu$, we get $\lambda = \langle \mathcal{T}\varphi, \varphi\rangle_\mu$.
  Concerning the upper bound in the second claim, using \eqref{energy}--\eqref{e-f-g-adjoint} we can derive that $0 \le \mathcal{E}(\varphi) =\langle (I-\mathcal{T})\varphi,\varphi\rangle_\mu = 1-\lambda$, which implies $\lambda \le 1$.
  If $\lambda=1$, we must have $\mathcal{E}(\varphi)=0$ and therefore \eqref{energy} implies that $\varphi$ is constant, since both $p$ and $\pi$ are positive.
  For the lower bound, similar to the proof of \eqref{e-f-g-adjoint} in Lemma~\ref{lemma-energy-and-msv}, we have (since $\mathcal{T}$ is self-adjoint)
  \begin{equation*}
    0 \le \frac{1}{2} \int_{\mathbb{R}^d}\left(\int_{\mathbb{R}^d}  (\varphi(y) +
    \varphi(x))^2 p(x,y)\dd{y}\right)\,\mu(\dd{x}) = \langle (I+\mathcal{T})\varphi, \varphi\rangle_\mu = 1+\lambda\,,
  \end{equation*}
  which implies $\lambda > -1$, since the equality is not attainable unless $\varphi\equiv 0$.
\end{proof}

\begin{proof}[Proof of Theorem~\ref{thm-variational-form-transfer}]
  Let $f_1, f_2, \dots, f_m\in \mathcal{H}$ be $m$ functions such that \eqref{f-constraints} holds.
  Define the diagonal matrix $A=\mbox{diag}\{\omega_1,\omega_2, \dots, \omega_m\}$
  and the $m\times m$ symmetric matrix $B=(B_{ij})_{1\le i,j\le m}$, whose
  entries are $B_{ij} = \mathcal{E}(f_i, f_j)$ (see \eqref{dirichlet-form}).
  Using the fact that $A$ is diagonal and $B_{ii} = \mathcal{E}(f_i,f_i) = \mathcal{E}(f_i)$ for $1 \le i \le m$ (see
  \eqref{energy}), we have 
  \begin{equation}
 \sum_{i=1}^m \omega_i \mathcal{E}(f_i) = \mbox{\textnormal{tr}} (AB) \,.
    \label{trace-formula}
  \end{equation}
  Also, we have 
  \begin{equation}
    c^\top B c = \sum_{1\le i,j\le m} c_i B_{ij} c_j 
= \sum_{1\le i,j\le m} c_i \mathcal{E}(f_i, f_j) c_j
     = \mathcal{E}\Big(\sum_{i=1}^m c_i f_i\Big) \ge 0\,,\quad \forall c \in \mathbb{R}^m\,,
    \label{cfc}
  \end{equation}
  and it equals zero only if $c=0$ (because \eqref{f-constraints} implies that
  $f_1,f_2,\cdots, f_m$ are linearly independent and $\mathcal{E}$ is defined in \eqref{energy} with positive $p$).
   This implies that $B$ is a positive definite symmetric matrix, whose
   eigenvalues are denoted by $0 < \widetilde{\lambda}_1 \le \widetilde{\lambda}_2 \le \dots \le \widetilde{\lambda}_m$. 
  Applying Ruhe's trace inequality~\cite[H.1.h, Section H, Chapter
  9]{marshall-inequalities-book} and using the assumption that $\omega_1 \ge \dots \ge \omega_m >0$,  we obtain from \eqref{trace-formula} that
  \begin{equation}
    \sum_{i=1}^m\omega_i \mathcal{E}(f_i) = \mbox{\textnormal{tr}} (AB) \ge \sum_{i=1}^m \omega_i \widetilde{\lambda}_i \,.
    \label{bound-1}
  \end{equation}
 Let us show that $\widetilde{\lambda}_j \ge 1-\lambda_j$ for $j\in
  \{1,2,\dots, m\}$. For this purpose, applying the min-max theorem for
  symmetric matrices we have 
  \begin{equation}
    \widetilde{\lambda}_j = \min_{S_j} \max_{c\in S_j, |c|=1}
    c^\top B c= \min_{S_j} \max_{c\in S_j, |c|=1} \mathcal{E}\big(\sum_{i=1}^m c_i f_i\big)\,,
    \label{lambda-k-tilde}
  \end{equation}
  where $S_j$ goes over all $j$-dimensional linear subspaces of $\mathbb{R}^m$ and
  the second equality follows from \eqref{cfc}. 
Since $(f_i)_{1\le i\le m}\subset \mathcal{H}$ satisfy \eqref{f-constraints},
  the functions $\mathbf{1}, f_1,\dots, f_m$ are pairwise orthonormal, where $\mathbf{1}$ denotes the constant function that equals one.
  Hence, each $j$-dimensional linear subspace $S_j\subset \mathbb{R}^m$ defines a
  $(j+1)$-dimensional linear subspace of $\mathcal{H}$ by $H_j=\big\{c_0 \mathbf{1} + \sum_{i=1}^m c_if_i\,|\, c_0\in \mathbb{R}, c\in S_j\big\}$, such that  $H_j\subset \mbox{span}\{\mathbf{1}, f_1,f_2, \dots, f_m\}$. 
  For $f = c_0 \mathbf{1} + \sum_{i=1}^m c_if_i$, the pairwise orthonormality
  and the fact $\mathcal{E}(f)=\mathcal{E}(f-c_0)$ (see \eqref{energy}) imply that $\|f\|_\mu=c_0^2 + |c|^2$
  and $\mathcal{E}(f) = \mathcal{E}\big(\sum_{i=1}^m c_i f_i\big)$. 
  Hence, we have 
  \begin{align*}
    \max_{f \in H_j, \|f\|_\mu=1} \mathcal{E}(f) =& \max_{\substack{c_0 \in
    \mathbb{R}, c\in S_j, c_0^2 + |c|^2 = 1}} \mathcal{E}(\sum_{i=1}^m c_i f_i)\\
    =& \max_{\substack{c\in S_j, |c| \le 1}} \mathcal{E}(\sum_{i=1}^m c_i f_i)
    = \max_{\substack{c\in S_j, |c| =1}} \mathcal{E}(\sum_{i=1}^m c_i f_i)\,,
  \end{align*}
  where the last equality follows from the fact that $\mathcal{E}$ is a (nonnegative) quadratic form.
 Therefore, from \eqref{lambda-k-tilde} we find that 
  \begin{align}
    \begin{split}
      \widetilde{\lambda}_j =& \min_{S_j} \max_{c\in S_j, |c|=1}
      \mathcal{E}\big(\sum_{i=1}^m c_i f_i\big)\\
      =& \min_{H_j} \max_{f \in H_j, \|f\|_\mu=1} \mathcal{E}(f) \\
      =& \min_{H_j} \max_{f \in H_j, \|f\|_\mu=1} \langle (\mathcal{I}-\mathcal{T})f, f\rangle_\mu \\
      \ge&  1 - \lambda_j\,,
  \end{split}
    \label{cmp-eigen-of-f-to-real}
  \end{align}
where the last inequality follows from the min-max Theorem~\cite[Theorem
  4.14]{teschl2009mathematical} and the fact that $1-\lambda_j$ is the $(j+1)$th smallest eigenvalue of $\mathcal{I}-\mathcal{T}$. 

  Combining \eqref{cmp-eigen-of-f-to-real} and \eqref{bound-1}, gives 
    $\sum_{i=1}^m\omega_i \mathcal{E}(f_i) \ge \sum_{i=1}^m \omega_i
    (1-\lambda_i)$.
    Since the eigenfunctions $(\varphi_i)_{1\le i\le m}$ satisfy \eqref{f-constraints}
    and $\sum_{i=1}^m\omega_i \mathcal{E}(\varphi_i) =
    \sum_{i=1}^m \omega_i (1-\lambda_i)$, we conclude that
    the identity \eqref{variational-transfer-all-first-m} holds and the minimum is achieved when $f_i=\varphi_i$ for $i\in \{1,2,\dots, m\}$.
\end{proof}

\begin{proof}[Proof of Proposition~\ref{prop-kab}]

  The first identity follows immediately from the definition \eqref{k-ab-def} and the observation that between two consecutive reactive segments from $A$ to $B$ there is exactly one reactive segment from $B$ to $A$.

  Concerning the second claim, we notice that the number of reactive segments within the first $N$ steps can be written as
  \begin{equation*}
    M_{N}^R = \sum_{m=0}^{N-1} \mathbbm{1}_A(X_m) \mathbbm{1}_B(X_{m+j_m})\,,
  \end{equation*}
  where $\mathbbm{1}_A$ (resp.\ $\mathbbm{1}_B$) denotes the indicator
  function of subset $A$ (resp.\ $B$), and $j_m:=\inf\{i \ge 1;  X_{m+i} \in A\cup B\}$.
  Therefore, by the definition \eqref{k-ab-def}, we have 
  \begin{equation*}
    k_{AB} = \lim_{N\rightarrow +\infty} \frac{1}{N} \sum_{m=0}^{N-1}
    \mathbbm{1}_A(X_m) \mathbbm{1}_B(X_{m+j_m}) =\int_{A}
    \left(\int_{\mathbb{R}^d} p(x,y) q(y)\, \dd{y}\right) \mu(\dd{x})\,,
  \end{equation*}
  where the second equality follows from the definition of the committor $q$ in \eqref{committor-eqn}, 
the ergodicity of the process $X_n$ and Birkhoff's ergodic theorem (see~\cite[Theorem~4.4 in Chapter 1]{Krengel+1985}). Similarly, by considering reactive segments from $B$ to $A$ and applying the identity in the first claim, we obtain the second equality in this claim, since 
  \begin{equation}
    \begin{aligned}
    k_{AB}=k_{BA}=& \lim_{N\rightarrow +\infty} \frac{1}{N} \sum_{m=0}^{N-1}
    \mathbbm{1}_B(X_m) \mathbbm{1}_A(X_{m+j_m}) \\
      =& \int_{B} \left(\int_{\mathbb{R}^d} p(x,y) (1-q(y))\,\dd{y}\right)\mu(\dd{x})\,.
    \end{aligned}
    \label{proof-prop-kab-1}
  \end{equation}
  Concerning the third claim, from \eqref{energy} we can compute
  \begin{align*}
    \mathcal{E}(q) =& \frac{1}{2} \int_{\mathbb{R}^d} \left(\int_{\mathbb{R}^d} |q(y) - q(x)|^2 p(x,y) \dd{y}\,\right)\mu(\dd{x})\\
    =& \langle q, (\mathcal{I}-\mathcal{T})q\rangle_\mu\\
      =& \int_{A} q(x) [(\mathcal{I}-\mathcal{T})q](x) \mu(\dd{x}) +
      \int_{(A\cup B)^c} q(x) [(\mathcal{I}-\mathcal{T})q](x) \mu(\dd{x})
      \\
      &+ \int_{B} q(x) [(\mathcal{I}-\mathcal{T})q](x) \mu(\dd{x}) \\
      =& \int_{B} [(\mathcal{I}-\mathcal{T})q](x)  \mu(\dd{x}) \\
      =& \int_{B} \left(\int_{\mathbb{R}^d} p(x,y) (1-q(y))\,\dd{y}\right) \mu(\dd{x})\\
      =& k_{AB}\,,
  \end{align*}
  where the second equality follows from \eqref{e-f-g-adjoint} in Lemma~\ref{lemma-energy-and-msv} and the fact that $\langle f,
  (\mathcal{I}-\mathcal{T}^{rev})f\rangle_\mu = \langle f,
  (\mathcal{I}-\mathcal{T})f\rangle_\mu$ for any $f$, the fourth equality
  follows by using the equation \eqref{committor-equation} of the committor
  $q$, the definition of $\mathcal{T}$ and the boundary conditions of $q$ on $A, B$, the
  fifth equality follows from the fact that $\int_{\mathbb{R}^d}
  p(x,y)\dd{y}=1$ and $q|_B=1$, and the last equality follows from \eqref{proof-prop-kab-1}.
\end{proof}

\bibliographystyle{plain}
\bibliography{references}

\end{document}